\documentclass{article}
\usepackage{tikz}
\usepackage{MyArxiv}

\usepackage[utf8]{inputenc} % allow utf-8 input
\usepackage[T1]{fontenc}    % use 8-bit T1 fonts
\usepackage{hyperref}       % hyperlinks
\usepackage{url}            % simple URL typesetting
\usepackage{booktabs}       % professional-quality tables
\usepackage{amsfonts}       % blackboard math symbols
\usepackage{nicefrac}       % compact symbols for 1/2, etc.
\usepackage{microtype}      % microtypography
\usepackage{lipsum}

\usepackage{amsmath}            %!
\usepackage{amssymb} 
\usepackage{csquotes}

\usepackage[normalem]{ulem}

\usepackage{todonotes}          %!
\usepackage{amsthm}
\usepackage[capitalize]{cleveref}
\usepackage{mathHeader}         %!
\graphicspath{{figures/}}
\def\input@path{{.}{figures/}}  %!
\usepackage{pgf}

\usepackage{hyphenat}
\usepackage{fancyhdr}       % header
\usepackage{graphicx}       % graphics
\graphicspath{{media/}}     % organize your images and other figures under media/ folder

\usepackage{floatrow}       % For captions to side of figures

\renewcommand{\d}{\mathrm{d}}	% Differential operator

\title{Sensitivities in complex-time flows: phase transitions, Hamiltonian structure and differential geometry 
}

\author{
  Dirk Lebiedz, Johannes Poppe \\
  Institute for Numerical Mathematics \\
  Ulm University \\
  \texttt{dirk.lebiedz@uni-ulm.de, johannespoppe92@gmail.com} \\
}

\begin{document}
\maketitle

\begin{abstract}
 Reminiscent of physical phase transitions separatrices 
    divide the phase space of dynamical
    systems with multiple equilibria into regions of distinct flow behavior 
    and asymptotics. We introduce complex time in order to study 
    corresponding Riemann surface solutions of holomorphic and 
    meromorphic flows, explicitly solve their sensitivity differential equation and identify a related Hamiltonian structure and an associated geometry in order to study separatrix properties. As an application we analyze complex-time Newton flow of Riemann's $\xi$-function on the basis of a compactly convergent polynomial approximation of its Riemann surface solution defined as zero set of polynomials, e.g. algebraic curves 
    over $\C$ (in the complex projective plane respectively), that is closely related to a complex-valued Hamiltonian system. Its geometric properties might contain information on the global separatrix structure and the root location of $\xi$ and $\xi'$.
\end{abstract}

% keywords can be removed


\section{Introduction}
Studying the topology of phase portraits of vector fields $\dot{x}=f(x), x\in \R^n$ might yield insight into properties of the generating function $f: \R^n \rightarrow \R^n$. If $n=2$ and $h:\C \ra \C \cong \R^2$ is an entire function,
\begin{equation}\label{eq:ode}
    \dot{z} = \dd[z]{t} = h(z), \quad z(0)=z_0 \in \C, t \in \R (\C)
\end{equation} has specific properties: Locally, the Cauchy-Riemann equation apply for real and imaginary part of $h$. Globally, the phase flow is usually divided into different stability regions by {\em separatrices} defining 'topological basins' around the roots of $h$. This is particularly interesting for flows of number{\hyp}theoretically
significant functions like the Riemann functions $h(z) \equiv \zeta(z)$ or $h(z) \equiv \xi(z)$ and the corresponding newton flows see e.g. \cite{Broughan2003, Broughan2003a,Neuberger2014,Schleich2018}
\begin{equation}\label{eq:newtonflow}
    z' = -\frac{h(z)}{h'(z)}, \quad z(0)=z_0 \in \C,
\end{equation}
with the roots of $h$ as asymptotically stable fixed points.

Broughan systematically analyzes and proves topological properties of holomorphic flows in a couple of interesting articles \cite{Broughan2003, Broughan2003a, Broughan2004, Broughan2005} that we refined and extended recently \cite{Kainz2023,Kainz2024,Kainz2024b}. Broughan refers to
the set composed of an equilibrium and its 'characteristic orbits' as the
\emph{neighborhood} of the corresponding equilibrium. Examples are a center
and all of its encircling periodic orbits or a stable node and its basin of attraction.  The boundaries of the neighborhoods of equilibria consist of separatrices, for which Broughan found an appropriate characterization which we address in the next section. 

As a
fundamental result it has been shown that holomorphic flows of type
(\ref{eq:ode}) do not allow limit cycles and sectors of the flow around
equilibria are confined by separatrices (see \cite{Broughan2003,Kainz2023}).  The recent
work of Schleich et al.~\cite{Schleich2018} suggests that the behavior of
$\xi$-Newton flow separatrices at infinity may contain cruciual information about the
location of the equilibria of such flows, e.g.\ the non{\hyp}trivial
zeros of $\xi(z)$.

A common feature of phase portraits is the occurrence of 'bundling behavior' (mutual repelling respectively) of trajectories near the separatrices, similar to the features of slow invariant manifolds (SIMs) and normally attracting invariant manifolds (NAIMs) (see e.g. \cite{Lebiedz2005b,Lebiedz2006a,Lebiedz2011a,Lebiedz2013a,Lebiedz2016,Heitel2019a,heiter2018towards,Lebiedz2022diff}). 

Here, we exploit experience and adopt tools from the latter context in order to study separatrix properties via time-propagation of perturbations $\Delta z$ based on sensitivity analysis. 
In case of holomorphic flows, there is an explicit solution formula $\Delta z(z(t))$ for the governing sensitivity equation (see \cite{Lebiedz2020}) evaluated along solution trajectories and linking the original ODE to the fields of Hamiltonian Systems (see section \Cref{sec:HamSens}) and Riemannian Geometry (see \Cref{sec:geomSens}). These links might provide toolboxes for extracting valuable global information about the phase portrait and the function $h$ and might also establish a link to chaotic systems which are discussed in the literature in the context of the Riemann hypothesis by various authors (see e.g. \cite{Berry1999}. In particular, separatrices play a crucial role in Hamiltonian chaos \cite{Zaslavsky2005h,Treschev2010i,Neishtadt1997s}.

\section{Separatrices and Complex Time}\label{sec:Seps&CompTime}

An appropriate mathematical
definition of separatrices is a controversial topic (cf. \cite{Broughan2003a}, the following \cref{def:sep} is reasonable and well suited for our particular purpose:
\begin{definition}[Separatrix \cite{Broughan2003a}]\label{def:sep}
    A trajectory $\gamma$ is a \emph{positive (negative) separatrix} for the
    holomorphic flow \eqref{eq:ode} if there is a point $z \in \gamma$ such that
    the maximum interval of existence of the path starting at $z$ and following
    the holomorphic flow \eqref{eq:ode} in positive (negative) time is finite. 
\end{definition}
This definition is not based on the separation of stability regions
in the phase space, but Broughan shows \cite{Broughan2003, Broughan2003a} that the
boundary of these stability regions consists of trajectories with finite
maximum interval of existence, and thus separatrices according to his definition.  

Since separatrices have finite maximum interval of existence, the vector field
$f$ cannot be bounded on separatrices. Otherwise, one could enlarge the
interval of existence. Therefore, a positive separatrix for entire functions
$f$ must reach out to infinity in phase space. This motivates the investigation
of the qualitative behavior of holomorphic flows at infinity which we analyze
in \cite{Heitel2019a}.

\begin{figure}[htbp]
    \centering
    \includegraphics[scale = 1.4]{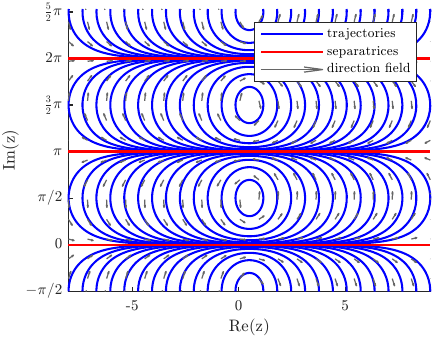}
    \caption{Phase portrait holomorphic flow $\dot{z}= \cosh(z-1/2)$}. %\enquote{Ordinary}
    \label{fig:cosh}
\end{figure}

A holomorphic flow example is shown in \cref{fig:cosh}, the $\cosh$ flow is related to the flow of Riemann's $\xi$ function (see \cite{Broughan2005}, section 6). Here, the lines $\Im z = k\pi, k \in \Z$ are special trajectories (separatrices in the sense of Def. \ref{def:sep}
separating distinct regions of stability (neighborhood of centers), the sets $S_k := \left\{z  \in \C \, : \, \Im
z \in \big(k\pi, (k+1)\pi\big) \right\}$ for $k\in \Z$. Within such a region
$S_k$, all trajectories are periodic orbits around a particular equilibrium $z_k^* :=
\frac{1}{2} + \left(k + \frac{1}{2}\right)\pi$. 
Since topological properties change discontinuously when crossing separatrices, these can serve as models for phase transitions. 
A direct calculation for $\cosh$ shows that the red trajectories have a bounded interval of existence in both time directions, i.e. they are both positive \emph{and} negative separatrices in the sense of Def. \ref{def:sep}.

Recently we proposed complex time and analytic continuation of real{\hyp}analytic dynamical systems to study the properties of slow invariant manifolds (SIM), whose phase space similarities to separatrices have been pointed out above (in particular trajectory bundling and attraction/repulsion in positive, respectively negative time direction). It became obvious that spectral information can be gained by studying imaginary time trajectories of linear and nonlinear systems. Fourier transform even yields quantitative criteria characterizing SIM in several example systems (see \cite{Dietrich2019a}).

The solution trajectories of (\ref{eq:ode}) in complex time $t \in \C$ are embedded Riemann surfaces (ho\-lo\-mor\-phic curves). For the complex time ODE 
\[ \dd{t}z(t) = \dd{(\tau_1+\rmi \tau_2)}z(t) = h\big(z(t)\big), z_0 \in \C \]
with holomorphic function $h$, there is at least a local, unique solution $z(\cdot)$ representing a holomorphic curve \cite{Ilyashenko2008}. This solution satisfies the following partial differential equations:
\begin{equation}\label{eq:complextimeODE}
    \pp{\tau_1}z(t) = h\big(z(t)\big), \quad  \pp{\tau_2}z(t) = \rmi h\big(z(t)\big), z_0 \in \C
\end{equation}

\begin{figure}[ht]
    \begin{center}
	\begin{minipage}{0.49\textwidth}
    \includegraphics{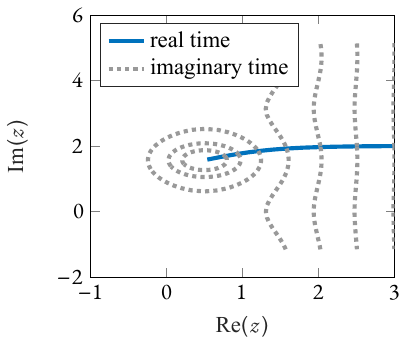}
	\end{minipage}		
	\begin{minipage}{0.49\textwidth}
     \includegraphics{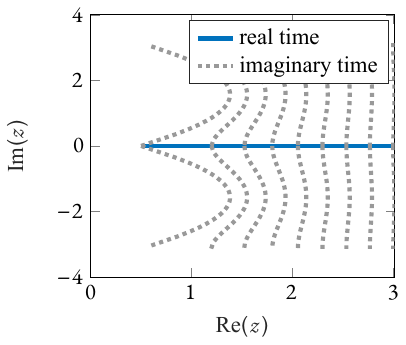}
	\end{minipage}
	\caption{Flow of $\cosh$ in real time (blue trajectories) and imaginary time (gray trajectories). Left figure: initial value of the real-time trajectory not on a separatrix. Right figure: initial value on a separatrix.}\label{fig:coshNewtonFlow2D}
    \end{center}
\end{figure}

In \Cref{fig:coshNewtonFlow2D} both real- and imaginary time trajectories for the $\cosh$ holomorphic flow are depicted. In the left plot, the blue line is some arbitrary non-separatrix real-time trajectory. On the right, the blue line is a separatrix. In both cases, the dotted imaginary time trajectories cross the real-time orthogonally. There are qualitative differences between the plots suggesting complex-time flows could be a helpful tool for separatrix phase transitions. We propse that this is an appropriate setting to study separatrices from a global point with tools from complex analysis, topology and algebra related to Riemann surfaces. 

\section{Newton Flow of Riemanns \texorpdfstring{$\xi$}{xi}-Function}\label{sec:NewtonFlow}
Newton flow lines $z(t)$ (solutions of (\ref{eq:newtonflow})) share the characteristic property of constant phase of the complex number $h(z(t))$ along solutions $z(t)$ and can be interpreted as gradient flow lines with respect to a Sobolev gradient defined by a Riemannian metric induced by the vector field \cite{Neuberger1999}.
Lines of constant modulus $|h(z(t)|$ are orthogonal to Newton flow lines. In a complex time view (see (\ref{eq:complextimeODE})), the latter correspond to imaginary time trajectories. Newton flows can be desingularized removing the singularities at zeros of $h'$ by rescaling time which results in the same flow lines with different time parameterization \cite{Benzinger1991,Jongen1988}.
The singularities of the Newton flow turn into hyperbolic fixed points (saddles). 

In particular, the inverted Newton flow of the Riemann $\zeta${\hyp}function occurs in the Riemann{\hyp}von{\hyp}Mangoldt explicit formula with $\zeta${\hyp}zeros $\rho_n$:

\begin{equation*}
    \psi_0(x) = \frac{1}{2\pi \rmi}\int\limits_{\sigma-\rmi \infty}^{\sigma +\rmi \infty}\left(-\frac{\zeta'(z)}{\zeta(z)}\right) \frac{x^z}{z} \dif z = x -\sum\limits_n \frac{x^{\rho_n}}{\rho_n} - \log(2\pi) -\frac{1}{2} \log(1-x^{-2}),
\end{equation*}
where
\begin{equation*}
    \psi(x) = \sum\limits_{p^k \le x} \log p \quad \mathrm{and} \quad \psi_0(x) = \frac{1}{2} \lim\limits_{h\rightarrow 0}\left(\psi(x+h)+\psi(x-h)\right)
\end{equation*}
with the second Chebyshev function $\psi$.

This was the original motivation for the authors to study a dynamical system based on $\zeta$ Newton flow and to introduce complex time. The Riemann{\hyp}von{\hyp}Mangoldt explicit formula contains the integral over the inverted $\zeta$ Newton flow in imaginary direction and relates this integral to a sum over $\zeta${\hyp}zeros. Thus the phase portrait of the $\zeta$ Newton flow, respectively \emph{symmetric} $\xi$ flow, should contain some global information on the $\xi$-zeros. 

We focus on the Newton flow of the symmetric (cf.\ \emph{functional equation}), entire $\xi${\hyp}function whose phase portrait has characteristic similarities to the complex cosh Newton flow.

\begin{equation*}
    \xi(z) := \frac{1}{2}z(z-1)\Gamma\left(\frac{z}{2}\right)\pi^{-z/2} \zeta(z).
\end{equation*}

Using the logarithmic derivative of the formula (see \cite[p.~47]{Edwards2001}) $\xi(z)=\xi(0) \prod\limits_{n}(1-\frac{z}{\rho_n})$ and separation of variables for the Newton Flow \eqref{eq:newtonflow} of $\xi(z)$ 
we get 
\begin{align*}
	& & -\frac{\xi'(z)}{\xi(z)} \dif z&=1 \cdot \dif t, & \left| \quad \int_{} \right.\\
	&\Leftrightarrow  & -\int\limits_{z_0}^{z(T)}\frac{\xi'(z)}{\xi(z)} \dif z &= \int\limits_{0}^{T} 1 \dif t, & z(0)=z_0\\
	&\Leftrightarrow & -\int\limits_{z_0}^{z(T)} \sum\limits_{n} \frac{1}{z-\rho_n}\dif z &= T &\\
	&\Leftrightarrow & -\sum\limits_{n}\int\limits_{z_0}^{z(T)} \frac{1}{z-\rho_n}\dif z &= T &\\
	&\Leftrightarrow & \sum\limits_{n}\log(z_0-\rho_n) - \sum\limits_{n} \log(z(T)-\rho_n) &\equiv T & (\mathrm{mod}\; 2\pi \rmi)\\
	&\Leftrightarrow & \sum\limits_{n} \log \frac{z(T)-\rho_n}{z_0-\rho_n} &\equiv -T & (\mathrm{mod}\; 2\pi \rmi) \\
	&\Leftrightarrow & \exp\left(\sum\limits_{n} \log \frac{z(T)-\rho_n}{z_0-\rho_n}\right) & \equiv \rme^{-T} & (\mathrm{mod}\; 2\pi \rmi) \\
	&\Leftrightarrow & \prod\limits_{n} \frac{z(T)-\rho_n}{z_0-\rho_n} &\equiv \rme^{-T} & (\mathrm{mod}\; 2\pi \rmi)
\end{align*}
When the integration path encloses zeros of $\xi$, the residue theorem has to be applied which brings addition of $k2\pi \rmi$. 
Finite approximation of the infinite product of $\xi${\hyp}zeros allows the definition of a polynomial
\begin{equation*}
    P_m(z;T,z_0):= \prod\limits_{n=1}^{m} \frac{z-\rho_n}{z_0-\rho_n} - \rme^{-T}.
\end{equation*}
An alternative derivation of $P_m(z;T,z_0)=$ is to combine the product formula
\begin{equation*}
    \xi(z) = \xi(z_0)\prod\limits_{n} \frac{z-\rho_n}{z_0-\rho_n}
\end{equation*}
with a line of constant phase as solution trajectory of the Newton flow of $\xi$
\begin{equation*}
    \xi\big(z(T)\big) = \xi(z_0)\rme^{-T}.
\end{equation*}
This also yields 
\begin{equation*}
    \rme^{-T} = \prod\limits_n \frac{z(T)-\rho_n}{z_0-\rho_n},
\end{equation*}
with approximating polynomials $P_m\big(z;T,z_0\big)$.

Contributions to the Newton flow of those terms corresponding to $\xi$ zeros \enquote{far away} from the initial value should be small for $|T|$ not too large.
The complex value of the solution trajectory at the specific time $T \in \C$ starting at $z(0)=z_0$ can then by approximated by the polynomial equation
\begin{equation}\label{eq:polynomapproxXi}
    P_m(z;T;z_0) = 0,
\end{equation}
which defines a complex algebraic curve.
The approximated Newton flow vector field is in this case a rational function and can be transformed to a polynomial by rescaling time as in \cite{Benzinger1991}. 
Figure~\ref{fig:xiApprox4} illustrates that with increasing number of complex conjugated $\xi$-zeros the polynomial flow solution manifolds comes closer to the $\xi$ Newton flow phase portrait in \cref{fig:xiNewtonFlow}. An approximation theoretical analysis is subject to future research. 
Figures~\ref{fig:xiApprox4} and \ref{fig:xiApprox40} were created with \textsc{MATLAB}. The first 100 $\xi${\hyp}zeros of Odlyzko \cite{Odlyzko} were used for numerical calculations. Equation \eqref{eq:polynomapproxXi} is solved for discrete time points $T$ for each initial value $z_0$ on a discretized grid of the compact set $[-7,8] \times [-1,30]\rmi$. \cref{fig:xiNewtonFlow} shows the phase portrait of the Newton flow of $\xi$. It was calculated using Mathematica.

\begin{figure}
    \begin{minipage}{0.49\textwidth}
	\centering
 \includegraphics[scale = 1.3]{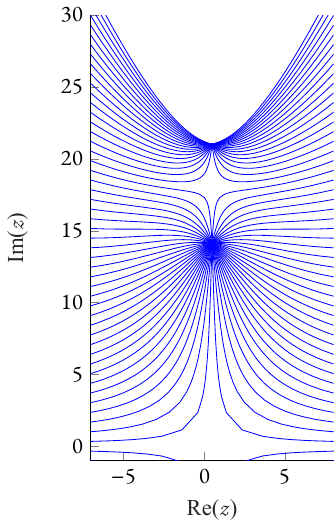}
	\caption{Appro\-xi\-ma\-tion of the $\xi$ Newton flow using \eqref{eq:polynomapproxXi}, $m=4$ (left) and $m=40$ (right) zeroes of $\xi$ (complex conjugate paits).}\label{fig:xiApprox4}
    \end{minipage}
    \hfill
    \begin{minipage}{0.49\textwidth}
	\centering
 \includegraphics[scale = 1.3]{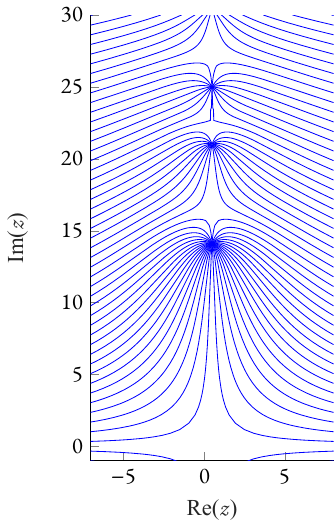}\label{fig:xiApprox40}
    \end{minipage}
\end{figure}   
\begin{figure}
    \centering
    \includegraphics[width=0.4\linewidth]{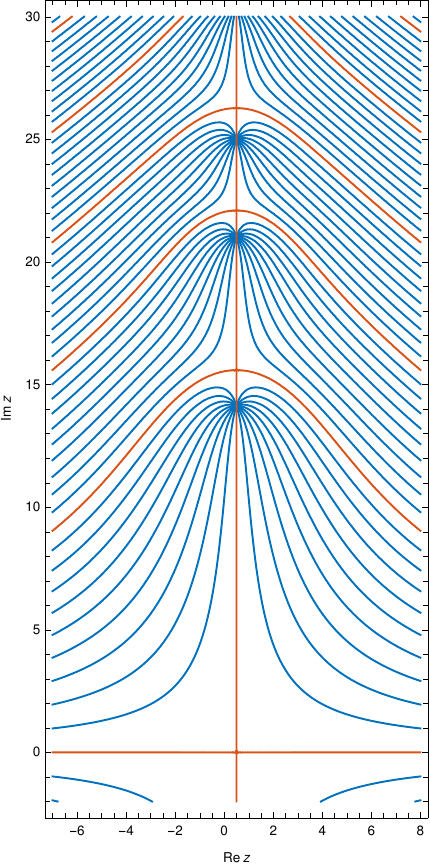}
    \caption{Newton flow of $\xi$. Trajectories are plotted by calculating lines of constant phase. Orange lines are separatrices.}\label{fig:xiNewtonFlow}
\end{figure}   

If $z(\cdot)$ is a periodic orbit of the Newton flow of $\xi$ in complex time with \enquote{complex period} $T \in \C$ a necessary condition for $z(T)=z_0 \in \C$ is
\begin{equation*}
    \rme^{-T} = \prod\limits_{n} \frac{z(T)-\rho_n}{z_0-\rho_n} = 1.
\end{equation*}
Therefore, $T$ must be a multiple of $2\pi \rmi$, i.e.\ $T = 2\pi k \rmi, k \in \Z$. 

We think that such polynomial approximations of the $\xi$ Newton flow based on an explicit formula for the $\xi${\hyp}zeros might allow to build a bridge into the world of algebraic geometry and corresponding approaches to the Riemann $\xi${\hyp}function and maybe even allow a view on the Riemann hypothesis. There might also be a relation to the Hilbert{\hyp}Polya conjecture associating the $\xi${\hyp}zeros with eigenvalues of some (self{\hyp}adjoint) operator. Note that the Riemann dynamics, a hypothetic classical chaotic dynamical systems which should lead to the Hilbert{\hyp}Polya operator via quantization, is supposed to have complex periodic orbits with periods $T_k=k\pi \rmi, k \in \N$ \cite{Berry1999,Sierra11}, which is in our case true for the Hamiltonian $H(z,p):=\xi(\frac{z}{2})p$. For the Hamiltonian dynamical system constructed from this type of Hamiltonian being closely related to the Newton flow see subsections \ref{subsec:HamSyst} and \ref{subsec:Xi_newt}.
%Due to successive Riemann surface branchings bifurcations of the Newton flow at zeros of $\xi'$, the dynamical systems in complex time should be chaotic in some sense with a specific bifurcation{\hyp}related route to chaos.
We speculate that in our setting an interesting operator might be constructed via the Newton flow map of $\xi$ related to a Hamiltonian system (see next subsection \ref{subsec:HamSyst}) mapping an initial function represented by imaginary time Newton flow solutions along positive real time (the modulus $|\xi(z)|$ is a conserved quantity along imaginary{\hyp}time trajectories of the Newton flow and the phase of the complex number $\xi(z)$ along real{\hyp}time trajectories!). In the physical literature a specific formal relation between statistical mechanics and quantum mechanics is referred to as Wick rotation \cite{Wick54}, sometimes called a mysterious connection: when replacing time in the evolution operator $\rme^{-\frac{\rmi}{\hbar}Ht}$ of some Hamiltonian $H$ (from time{\hyp}dependent Schr\"odinger equation) by negative imaginary time $t \rightarrow -\rmi t$ and then $\frac{t}{\hbar}$ by $\frac{1}{T}$ one ends up with the time{\hyp}independent state density operator $\rme^{-\frac{H}{T}}$ from equilibrium statistical mechanics. We propose that this view might be of particular relevance in our context. Looking at \cref{fig:xiApproxImag20} it becomes obvious how cyclic imaginary time trajectories are transported in real time direction along the Newton flow. When crossing in this time evolution a zero of the derivative $P_m'(z;T,z_0)$ on a separatrix, the Riemann solution surface in complex time bifurcates and in imaginary time direction two cyclic branches occur, one enclosing a single zero and the other enclosing all remaining zeros. This is reminiscent of a phase transition. 

In addition to the obvious algebraic issue to study the interesting distribution of zeros of the polynomials $P_m(z;T,z_0)$ (red dots in \cref{fig:xiApproxImag20}) as a function of the polynomial degree and their behavior at Riemann surface branching points, a more general relation to algebraic geometry is based on the fact that compact Riemann surfaces $\mathcal{R}$ are biholomorphically equivalent to complex projective algebraic curves. 
%The latter are also defined as curves over some number field under the conditions of Belyi's theorem \cite{Belyi79, Belyi02}, 
%i.e.\ a Riemann surface is defined as an algebraic curve over a finite field if and only if there is a non{\hyp}constant meromorphic function on $\mathcal{R}$ which branches in at most three points,
%i.e.\ $\mathcal{R}$ is a curve defined over the algebraic numbers $\bar{\mathbb{Q}}$  if and only if a morphism $\mathcal{R} \rightarrow \mathbb{P}^1_{\C}$ with at most three critical values exists. This establishes a link between topological coverings and field extensions. With a view on the bifurcations/ramifications of Riemann surfaces pointed out above in the context of complex{\hyp}time Newton flows (see \cref{fig:coshNewtonFlow2D} and \cref{fig:coshNewtonFlow4D}), it might turn out to be particularly fruitful to look at these bifurcations with the spirit of Belyi's theorem. 
%The role of the Belyi functions could be taken by inverted mappings of complex time to the Newton flow solution manifold, which is a covering of compact sets in the complex time plane.

%Another interesting direction could be to establish a link between the dynamical system that we suggested, i.e.\ the $\xi$ Newton flow in complex time and its polynomial approximations based on the $\xi${\hyp}zeros to the work of Deninger \cite{Deninger1998} investigating analogies between number theory and dynamical systems.

\begin{figure}
    \begin{minipage}{0.47\textwidth}
	\centering
	\includegraphics[width=0.9\linewidth]{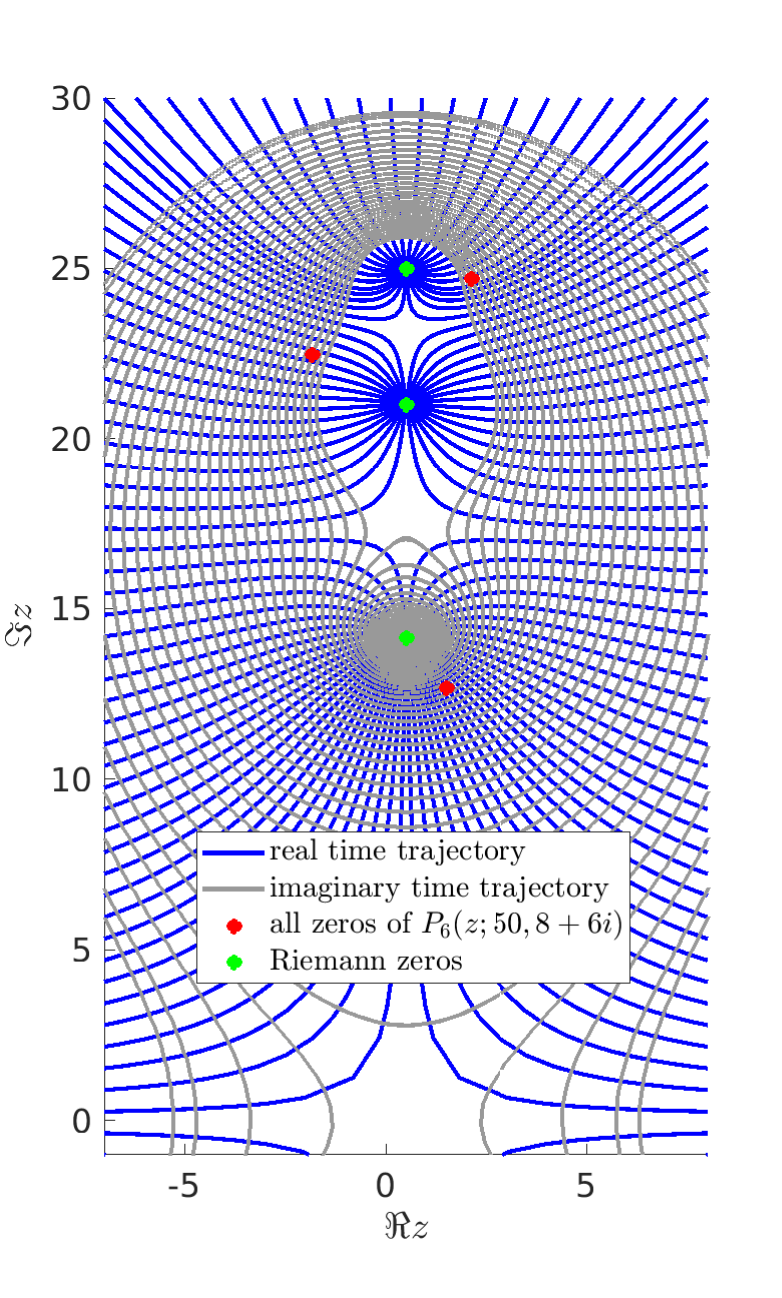}
	%\caption{Left: Real and imaginary time $\xi$ Newton flow solution ap\-prox\-i\-mat\-ed by \eqref{eq:polynomapproxXi} using $m=6$ zeros of $\xi$ (first three zeros and their complex conjugates). Right:Real and imaginary time $\xi$ Newton flow solution ap\-prox\-i\-mat\-ed by \eqref{eq:polynomapproxXi} using $m=20$ zeros of $\xi$ (first 10 zeros and their complex conjugates). }\label{fig:xiApproxImag6}
    \end{minipage}
    \hfill
    \begin{minipage}{0.47\textwidth}
	\centering
	\includegraphics[width=0.9\linewidth]{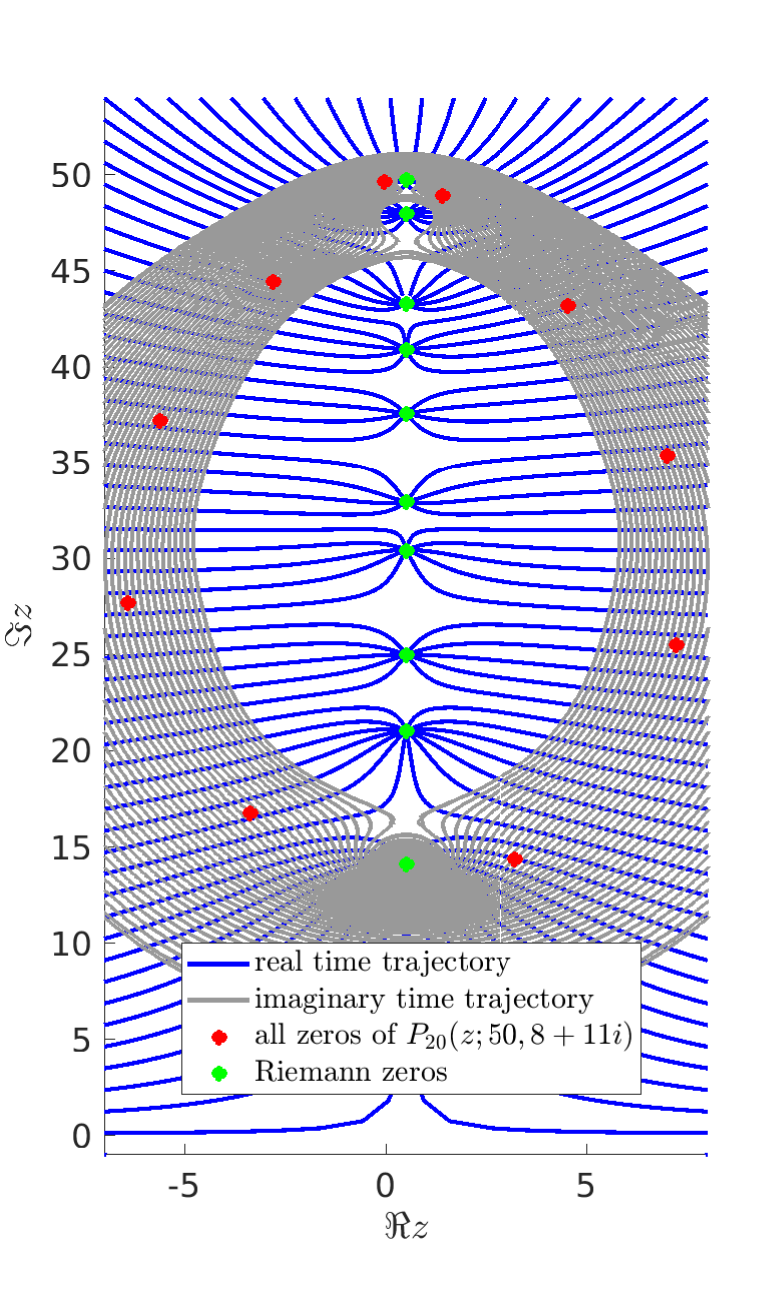}
	\caption{Left: Real and imaginary time $\xi$ Newton flow solution ap\-prox\-i\-mat\-ed by \eqref{eq:polynomapproxXi} using $m=6$ zeros of $\xi$ (first three zeros and their complex conjugates). Right: Real and imaginary time $\xi$ Newton flow solution ap\-prox\-i\-mat\-ed by \eqref{eq:polynomapproxXi} using $m=20$ zeros of $\xi$ (first 10 zeros and their complex conjugates).}\label{fig:xiApproxImag20}
    \end{minipage}
\end{figure}   

\subsection{Hamiltonian System}\label{subsec:HamSyst}
We propose a Hamiltonian system $z,p\in \C, H:\C^2 \rightarrow \C, H:=\xi(z)p, t \in \R$ which connects  the holomorphic $\xi$-flow with both its variational flow and Newton flow.  
\begin{eqnarray}\label{hamiltonsyst}
\dot{z}=\frac{\d z}{\d t}&=&\frac{\partial H}{\partial p}=\xi(z), \quad z(0)=z_0 \nonumber \\
\dot{p}=\frac{\d p}{\d t}&=&-\frac{\partial H}{\partial z}=-\xi'(z)p, \quad p(0)=p_0. 
\end{eqnarray}
Phase space formulation with momentum $p$ leads to 
\begin{equation}\label{phasespaceODE}
\frac{\d z}{\d p}=\frac{\dot{z}}{\dot{p}}=-\frac{\xi(z)}{\xi'(z)}\cdot\frac{1}{p} 
\Leftrightarrow
\frac{\xi'(z)}{\xi(z)} \d z =- \frac{1}{p} \d p.
\end{equation}
\noindent
Separation of variables and use of the logarithmic derivative $(\ln \xi(z))'=\frac{\xi'(z)}{\xi(z)}$ of the $\xi$-product formula 
\begin{equation}\label{xiprodformel}
\xi(z)=\xi(0) \prod\limits_{n}\left(1-\frac{z}{\rho_n}\right)=\xi(0)\prod\limits_{n}\left(\frac{\rho_n-z}{\rho_n}\right)
\end{equation} 
with the $\xi$-zeros $\rho_n\in \C$ yields 
\[ \ln \xi(z)=\ln \xi(0)+ \sum_{n\in\N}\ln \left(\rho_n-z\right)- \sum_{n\in\N}\ln \rho_n \]
and $(\ln \xi(z))'= \sum_{n\in\N} \frac{1}{z-\rho_n}$. For $z(0)=z_0, p(0)=p_0$ integration of the Hamiltonian system (\ref{phasespaceODE}) leads to
\begin{equation}
\label{xiapprox}
    \begin{array}{rrlr}
	& \frac{\xi'(z)}{\xi(z)} \d  z&=-\frac{1}{p} \d  p  &\\
	\Leftrightarrow  & \int\limits_{z_0}^{z}\frac{\xi'(s)}{\xi(s)} \d  s &= -\int\limits_{p_0}^{p} \frac{1}{s} \d  s & \\
	\Leftrightarrow & \int\limits_{z_0}^{z} \sum\limits_{n} \frac{1}{s-\rho_n}\d  z &= \ln \frac{p_0}{p} &\\
	\Leftrightarrow & \sum\limits_{n}\int\limits_{z_0}^{z} \frac{1}{s-\rho_n}\d  z &=  \ln \frac{p_0}{p}  &\\
	\Leftrightarrow & \sum\limits_{n}\ln(z-\rho_n) - \sum\limits_{n} \ln(z_0-\rho_n) &=  \ln \frac{p_0}{p} &\\
	\Leftrightarrow & \sum\limits_{n} \ln \frac{z-\rho_n}{z_0-\rho_n} &=  \ln \frac{p_0}{p} & \\
	\Leftrightarrow & \exp\left(\sum\limits_{n} \ln \frac{z-\rho_n}{z_0-\rho_n}\right) & =   \frac{p_0}{p} \quad \mathrm{(mod } \ 2\pi i)&\\
	\Leftrightarrow & \prod\limits_{n} \frac{z-\rho_n}{z_0-\rho_n} &= \frac{p_0}{p} \quad \mathrm{(mod } \ 2\pi i)& 
    \end{array}
\end{equation}

\subsection{Algebraic Geometry Viewpoint}
\noindent 
Based on (\ref{xiapprox}) we define a polynomial $P_m \in \C[z,p],m \in \N$
\begin{eqnarray}
P_m (z,p,z_0,p_0):= \nonumber \\ 
p \prod\limits_{n=-m,n\ne 0}^m (z-\rho_n) - p_0\prod\limits_{n=-m,n\ne 0}^m (z_0-\rho_n)
%= p \prod\limits_{n=1}^m (z-\rho_n) (z-\bar{\rho}_n) - A \\
%A:=p_0\prod\limits_{n=1}^m (z_0-\rho_n) (z_0-\bar{\rho}_n), m \in \N \nonumber 
\end{eqnarray}
with the product from $-m$ to $m$ referring to $m$ zeros together with their $m$ complex conjugated $\xi$-zeros.
% $p$ and $p_0$ are coupled via the solution $p(t)=-\frac{\xi(z(t))}{\xi(z_0)}p_0$  of the momentum equation in (\ref{hamiltonsyst}). 
We propose to consider the polynomial $P_m$ in complex projective space, which is for fixed initial values $(z_0,p_0) \in \C^2$ the projective plane $P^2_{\C}$ and the zero set $P_m(z,p)=0$ describes a plane algebraic curve on the Riemann sphere. Due to convergence of the $\xi$-product formula (\ref{xiprodformel})  (see \cite{Edwards2001} chap. 2 and \cite{Hadamard1893,VonMangoldt1895}), for $m\rightarrow \infty$ this algebraic set converges towards to solution manifold of the Hamiltonian flow (\ref{hamiltonsyst}).\\
The analytic and algebraic study of this manifold and its topology and symmetries might provide further insight into the properties of the $\xi$-function and the role of the Riemann zeros and their location. The work of Schleich et al. points out a significant role of the asymptotics of particular $\xi$-Newton flow lines, the separatrices (solution trajectories with a finite interval of existence, see \cite{Broughan2003a}), for the location of the Riemann zeros. In \cite{Heitel2019a} we suggest Poincar\'{e} sphere compactification of the dynamical system in order to study separatrices 'at infinity' and prove a result for the relation between fixed points and separatrices. We propose that the algebraic geometry viewpoint suggested here offers tools to characterize the separatrices as algebraic curves with particular geometric properties.

\subsection{Relation to \texorpdfstring{$\xi$}{xi}-Newton Flow}\label{subsec:Xi_newt}
\noindent
For momentum $p=p_0 e^T, T\in \C$ this result is in full agreement with the results from \cite{Heitel2019a} for the $\xi$-Newton flow in complex time. The solution of $\dot{p}= -\xi'(z)p$ is $p(t)=p_0e^{-\int_0^t \xi'(z(\tau)) \d \tau}=p_0e^{\xi(z_0)-\xi(z)}$, i.e. $T= \xi(z_0)-\xi(z) +2 \pi k i, t\in \mathbb{R}, T \in \C, k\in \mathbb{Z}$. $z(T)$ can then be interpreted as the solution of a Newton-flow equation in complex time:
\begin{eqnarray}
\frac{\d z}{\d T} &=& \frac{\d z}{\d p}  \cdot \frac{\d p}{\d T}  = -\frac{\xi(z)}{\xi^\prime(z)} \nonumber \\
\frac{\d  T}{\d t} &=& \frac{\d  T}{\d p} \cdot \frac{\d  p}{\d t} =  \frac{1}{p} \cdot (-\xi^{\prime}(z)p) = - \xi^{\prime}(z). 
\end{eqnarray}
With substitution and logarithmic derivative we get 
\begin{eqnarray}
 \int\limits_0^t \xi^\prime(z(\tau))\d \tau &=& \int\limits_{z_0}^z \frac{\xi^\prime(s)}{\xi(s)}\d s  =
\ln(\xi(z))-\ln(\xi(z_0)) \nonumber \\
&=& \ln \left(\frac{\xi(z)}{\xi(z_0)} \right),
\end{eqnarray}
and for the time $T = -\ln(\xi(z(t)) + \ln(\xi(z_0)) + 2\pi k i$. The function $T(t)$ solves the differential equation
\begin{eqnarray}
\frac{\d T}{\d t} &=& -\frac{\d }{\d t} \ln (\xi(z(t)) = - \frac{1}{\xi(z(t))}\xi^\prime(z((t))\xi(z(t)) \nonumber \\
&=&  - \xi^\prime(z(t)).
\end{eqnarray}
This is obviously a nonlinear (differential) reparameterization of time $t\in \R \rightarrow T\in \C$ via $\xi'$ along solutions of the $\xi$-flow.

\subsection{Hamiltonian Variational Equation and Stability}\label{subsec:HamStability}
\noindent
The propagation of perturbations of initial values $\Delta z_0, \Delta p_0$ along closed Hamiltonian orbits is described to first order by the variational differential equation
\begin{eqnarray}
 \frac{\d}{\d t} \Delta z&=&\xi'(z) \Delta z, \Delta z(0)=\Delta z_0 \label{vardgl1} \\
 \frac{\d}{\d t} \Delta p&=&-\xi''(z)p \Delta z - \xi'(z) \Delta p, \Delta p(0)=\Delta p_0 
\end{eqnarray} 
Over the field $\C$ eq. (\ref{vardgl1}) and the momentum equation of (\ref{hamiltonsyst}) are 1-dim. and can be explicitly solved to
\begin{equation}\label{eq:XiImpsol}
\Delta z= \frac{\xi(z)}{\xi(z_0)}\Delta z_0, \quad p= \frac{\xi(z_0)}{\xi(z)}p_0, \quad p\Delta z=  p_0 \Delta z_0
\end{equation} 
and so the $\Delta p$ equation is
\begin{equation}\label{eq:DqXiODE}
 \frac{\d}{\d t} \Delta p=-\xi''(z) \Delta z_0 p_0  -\xi'(z) \Delta p
\end{equation} 
Variation of constant $c(t)$ with ansatz $\Delta p :=c(t) \xi^{-1}(z)$ leads to the solution
\begin{equation}\label{eq:vardglsol}
\Delta p =  \frac{p_0\Delta z_0(\xi'(z_0)-\xi'(z))+\xi(z_0)\Delta p_0}{\xi(z)}
\end{equation} 
which turns out to be a summation formula using the $\xi$ product representation (\ref{xiprodformel}) and its logarithmic derivative:\begin{equation}\label{traceformula}
\Delta p =  p_0 \Delta z_0 \left( \frac{\xi'(z_0)}{\xi(z)} - \sum\limits_{n} \frac{1}{z-\rho_n}\right) + \frac{\xi(z_0)}{\xi(z)} \Delta p_0,
\end{equation} 
The perturbation $\Delta p$ propagates with reference to the Riemann zeros. Eq. (\ref{traceformula}) is a sensitivity equation allowing local stability analysis for Hamiltonian orbits and contains a 'spectral summation' reminiscent of a classical trace formula \cite{Cvitanovic1999}. 
For the complex-valued flow map differential matrix $M\in\C^{2\times 2}$ describing the evolution $(\Delta z, \Delta p)^{\top}=M(\Delta z_0, \Delta p_0)^{\top}$  of an initial perturbation along a given Hamiltonian orbit $(z(t),p(t))\in\C^2$ we get 
\begin{equation}\label{diffpoincaremap}
 M=
  \left( {\begin{array}{cc}
    \frac{\xi(z)}{\xi(z_0)} &  0\\
  p_0\frac{\xi'(z_0)}{\xi(z)}- \sum\limits_{n} \frac{p_0}{z-\rho_n} & \frac{\xi(z_0)}{\xi(z)}  \\
  \end{array} } \right)
\end{equation}
with a coupling constant $k_{z,p}:= p_0\frac{\xi'(z_0)}{\xi(z)}-\sum\limits_{n} \frac{p_0}{z-\rho_n}$ between $z$- and $p$-space. 
%and 'resonances' at the Riemann zeros. 
$M$ takes a particularly simple form for initial values in the set $\{ z_0 \in \C:\xi'(z_0)=0\}$, which are on separatrices of the $\xi$-Newton flow \cite{Neuberger2014,Neuberger2015}. 
%Since $\lim_{z \rightarrow \rho_n} \Delta p = \infty$ in eq. (\ref{traceformula}) the corresponding evolution mapping $(\Delta z_0, \Delta p_0)\rightarrow (\Delta z, \Delta p)$ has singularities at $z=z_0=\rho_n$, the Riemann zeros missing in the continuum as in absorption spectra (see e.g. \cite{Connes1996} for an absorption spectrum interpretation of the Riemann zeros).

\subsection{Action as Periodic Orbit Transit Time}
\noindent
For the action along a periodic Hamiltonian orbit we compute $S_p(E)=\oint p \ \d z$ with constant flow-invariant Hamiltonian (energy) $E=H(p,z)=H(p_0,z_0)$:
\begin{equation}
S_p(E)= \oint p(t) \xi(z(t)) \d t=\oint H(z,p) \d t=H(z_0,p_0) \cdot t^*
\end{equation}
with closed orbit period $t^*$ of the $\xi$-flow (see \cite{Broughan2005}). The energy-time duality formula (2.11) $t^*_p=\frac{\partial S_p}{\partial E}$ from \cite{Berry1999} gives in our case the period $t^*=t^*_p=\frac{2\pi i}{\xi'(\rho_n)}$ with $n\in \N$ depending on $z_0$ fixing the periodic orbit. If the $\xi$-zero is simple, it is a center-type fixed point of the dynamical system \cite{Broughan2003}.
On a fixed energy level set invariant under the Hamiltonian flow the action is quantized. The quantization is determined by the derivatives $\xi'(\rho_0)$ at the Riemann zeros $\rho_n$ defining by elementary complex calculus \cite{Broughan2003,Broughan2005} the closed orbit period which is invariant under homotopy of periodic orbits. Broughan and Barnett already pointed out a potential role of $\xi$-flow closed orbit periods for the Hilbert-P\'{o}lya approach and numerically investigated a linear-scaling law for the periods with increasing imaginary part of the critical zeros with an upper bound for the periods \cite{Broughan2006lin} under the Gonek conjecture assuming a lower bound for the $\zeta$-derivative $\zeta'(\rho_n)$ evaluated at the Riemann zeros.

\section{Sensitivity of Polynomial Flows and Hamiltonian Structure}\label{sec:HamSens}

The Hamiltonian system in \Cref{subsec:HamSyst} is constructed by taking the Riemann $\xi$-function and multiplying it with a complex-valued variable $p$. We create an entire class of Hamiltonian functions $H(z,p) := h(z)p$ by replacing the $\xi$-function by a general holomorhic function $h$.

According to the considerations of \Cref{sec:NewtonFlow}, evident candidates to generate such Hamiltonians are the approximating polynomials based on the $\xi$-product formula 
\begin{equation}\label{eq:XiApprPoly}
    z \mapsto h_{2m}(z):= \xi(z_0) \prod_{n=-m, n \ne 0}^m \left(\frac{z-\rho_n}{\rho_n} \right)
\end{equation}
symmetrically (with respect to the real axis) containing the $m$ pairs of roots $\rho_n$ of $\xi$ on the critical axis with the smallest imaginary parts. We call them $\xi$-\textit{approximating polynomials} of (even) degree $2m$. In \Cref{fig:xiApprox4}, we have exemplarily shown the phase-space portraits of newton-flows of $h_4$ and $h_{40}$.

In the following sections, we point out our arguments for these polynomials ($m=4$, degree $=8$). For numerical reasons, we multiply these functions by a suitable constant $\alpha \in \R^{>0}$. The following solution formulas imply that this operation does not change the phase-space portrait or the shape of the resulting Riemann surfaces, only the velocity of the flow. 

% A number of calculations from \Cref{sec:NewtonFlow} can be generalized and applied to this class.

%In \Cref{subsec:HamSyst} we propose a specific

Given a holomorphic function $h$ and its Hamiltonian function $H(z,p) = h(z)p$, the corresponding Hamiltonian flow is
\begin{equation}\label{eq:HamFlow}
    \Dot{z}=h(z),\qquad \Dot{p} = -h'(z)p.
\end{equation}
The dynamics of the state variable $z$ is modeled by the holomorphic flow of $h$. As in classical Hamiltonian mechanics we call $p$ the momentum variable. We show how this system connects the holomorphic flow and its sensitivity by this momentum $p$. 

Since the Hamiltonian function $h(z)p$ is constant along trajectories, we can directly deduce the unique solution formula for the momentum $p$ with respect to $z$ 
\begin{equation} \label{eq:ImpSol}
    p = \frac{h(z_0)}{h(z)} p_0
\end{equation}
for a given pair of initial values $(z_0,p_0)$ (compare to \cref{eq:XiImpsol} for the $\xi$-function). Naturally, $p(z;z_0,p_0)$ also solves \cref{eq:HamFlow} as a function of the complex time $t$.  

%along the (real-time) solution curve $z(t)$ of the flow. In case of imaginary time, applying the aforementioned formula to the function $ih(z)$, solves the initial value problem \cref{eq:HamFlow} as well. We directly receive complex-time path-independence and \cref{eq:ImpSol} holds on all path-connected parts of the domain of the holomorphic flow $h$. 

\begin{figure}[hbtp]
    \centering
    \includegraphics[scale=0.35]{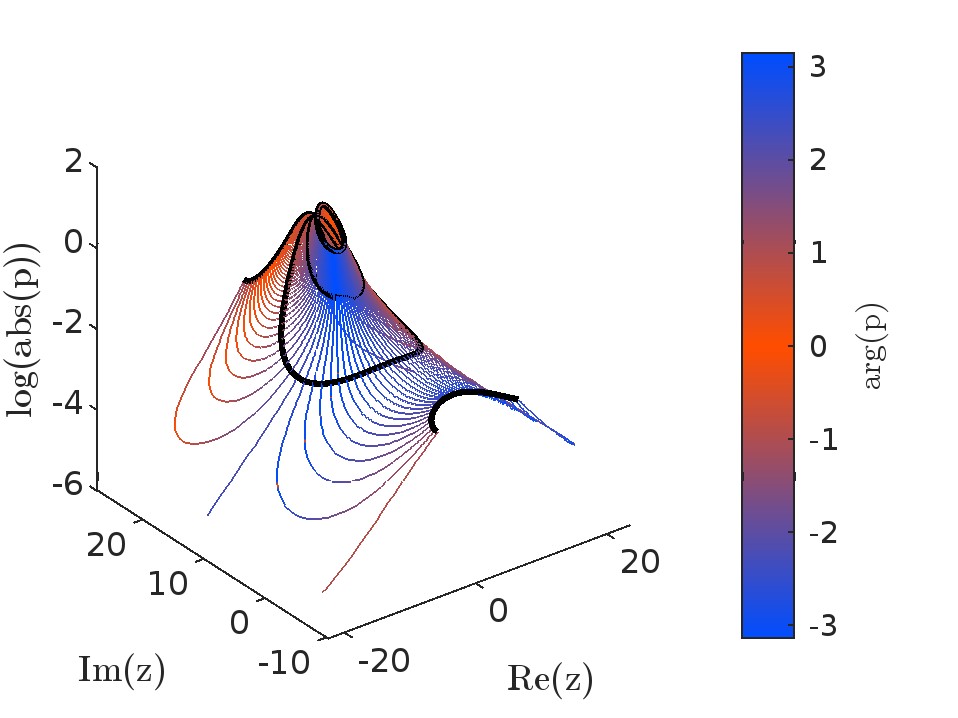}
    \caption{Solution surface of Hamiltonian system for the polynomial $h_8$. Colored lines are imaginary-time trajectories. Real-time trajectories are represented by black lines.}
    \label{fig:qp_surf}
\end{figure}

In \Cref{fig:qp_surf}, we show a solution of the Hamiltonian system for $h_{8}$ - the $\xi$-approximating polynomial of degree $8$. Here, each real-time (non-separatrix) solution trajectory of the state variable $z(t)$ is a periodic orbit around a single simple root of $h_8$. Visualizing \emph{one} solution of the Hamiltonian system, we use one pair on initial values $(z_0,p_0)$, create an equidistant  grid of real- and imaginary time lines and solve the discretized system numerically.

The solution of the momentum as a function depending on $z$ is divided into its absolute value (vertical axis) and its argument. The graph of this function has the structure of a Riemannian surface embedded in $\C^2$. On each surface, the Hamiltonian function $H$ is constant.

\subsection*{Sensitivity}
The propagation of perturbations of initial values $\Delta z_0$ along trajectories is described to first order by the variational differential equation
\begin{equation}\label{eq:Sens}
    \dd[\Delta z]{t} =h'(z) \Delta z, \quad \Delta z(0)=\Delta z_0 
\end{equation}
 Analogous to \cref{eq:ImpSol}, we receive the solution
\begin{equation} \label{eq:SensSol}
    \Delta z = \frac{h(z)}{h(z_0)} \Delta z_0 = h(z) \frac{\Delta z_0}{h(z_0)}
\end{equation}
for given initial values $z_0,p_0 \in \C$ (compare to \cref{eq:XiImpsol}). We see that the solution $\Delta z(z)$ behaves like the function $h(z)$ itself. In contrast, the momentum $p$ behaves like $1/h(z)$. This 'inverse connection' between $p$ and $\Delta z$ is also reflected in the fact that the ODEs \cref{eq:HamFlow} and \cref{eq:Sens} exactly match up to the reverse sign. Moreover, we see that
\[
\Delta z \cdot p = \frac{h(z)}{h(z_0)}\Delta z_0 \frac{h(z_0)}{h(z)} p_0 = \Delta z_0 \cdot p_0.
\]
A direct calculation shows that the function mapping $(z,p)$ to $\Delta z \cdot p$ is equivalent to the proposed Hamilton function $H$ (as long as $h(z)$ and $h(z_0)$ do not vanish). 

\begin{figure}[hbtp]
    \centering
    \includegraphics[scale=0.35]{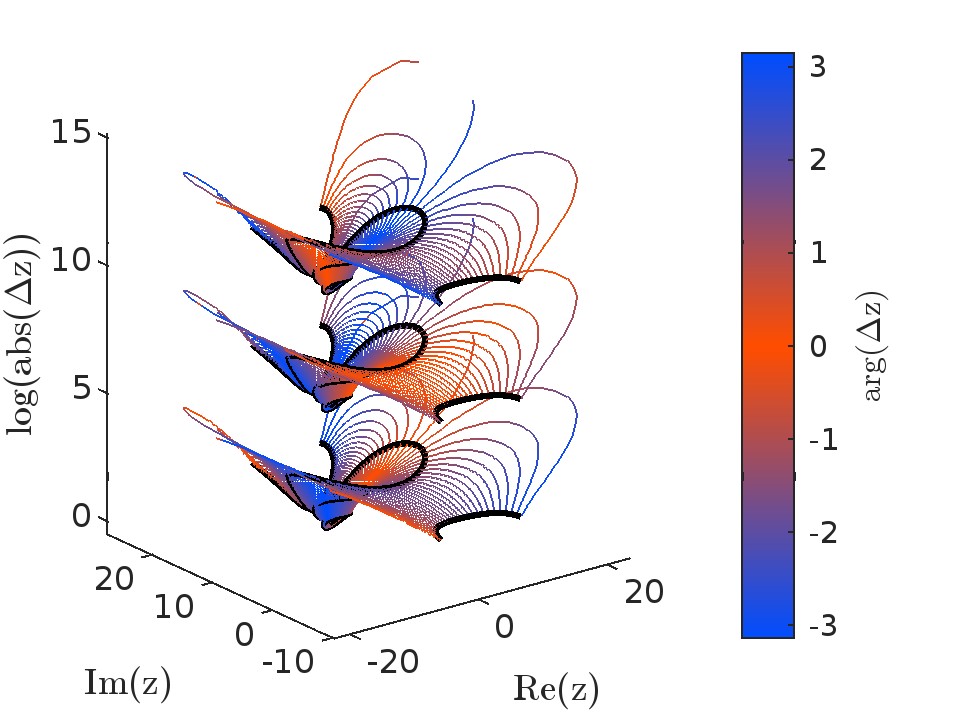}
    \caption{Solution surfaces of the sensitivity equation \cref{eq:Sens} for different initial values for the same polynomial $h_8$ as in \Cref{fig:qp_surf} for different initial values. Black- and colored lines represent real- and imaginary time trajectories respectively.}
    \label{fig:qdq_surf}
\end{figure}

In \Cref{fig:qdq_surf} we exemplify, how the sensitivity $\Delta z$ behaves along complex-time solutions of a holomorphic flow and visualize its dependence on the chosen initial value. For that we chose three different pairs of initial values $(z_0, \Delta z_0)$ and solve for the same complex-time numerical grid as in \Cref{fig:qp_surf}.  

The linear dependence of $\Delta z$ on $\Delta z_0$ is reflected in the fact that the resulting surfaces the same shape resulting from the scaling of the vertical axis and the properties of the logarithm implying
\[
\log(\text{abs}(\Delta z)) = \log(\text{abs}(\Delta z_0))+ \log\left(\text{abs}\left(\dfrac{h(z)}{h(z_0)}\right) \right).
\]

The 'inverse connection' to $p$ is also reflected by comparing this plot to \cref{fig:qp_surf}, where the surfaces appear to be 'upside down'. This is also due inverse nature of $p$ and $\Delta z$ and the way we scaled the vertical axis.

\subsection*{Connection between holomorphic flow and newton flow}
 
One central propertie of the newton flow (see \cref{eq:newtonflow}) is the fact that its solutions in real time conserve the phase of $h$. A direct calculation (for the complex time variable $t$) shows
\[
    \dd[h(z(t))]{t} = \pp[h(z)]{z} \dd[z]{t} = h'(t) \left( \frac{-h(z(t)}{h'(z(t)} \right) = -h(z(t)).
\]
This means that if we interpret $\Tilde{h} :=h \circ z$ as a function of $t$, it satisfies the ODE $\Tilde{h}' = -\Tilde{h}$. For a given initial point $z_0 = z(t_0)$ we receive the solution formula
\begin{equation}\label{eq:NewtonProp}
    h(z(t)) = h(z_0) e^{-t}, \quad t \in \C
\end{equation}
This is especially interesting when extracting solutions in real- and imaginary time:
\begin{remark} \label{rem:SensNewt}
    The real-time solutions of the newton flow conserve the argument of $h$ and therefore of the sensitivity $\Delta z$ of the holomoprhic flow. The imaginary-time solution are contour lines of $|h|$ - and therefore also of the absolute value of the sensitivity of the holomorphic flow.
\end{remark}
\begin{remark} \label{rem:ImpNewt}
    The real-time solutions of the newton flow conserve the argument of $1/h$ and therefore of the solution of the momentum $p$ in \cref{eq:ImpSol}. The imaginary-time solution are contour lines of $|1/h|$ - and therefore also of the absolute value $p$.
\end{remark}

Understanding the connection between a holomorphic flow and its newton flow might be valuable in order to study the underlying function $h$ itself.
\Cref{subsec:Xi_newt} establishes a link by calculating the complex derivative $\d t / \d T$ - where $t$ and $T$ are the complex time variables reserved for the holomorphic- and newton flow respectively - for the $\xi$-function. These arguments can be applied to any holomorphic function $h$, yielding
$\d t / \d T = -h'(z)$.
\Cref{rem:SensNewt,rem:ImpNewt} formulate how $\Delta z$ and $p$ are connected in both flows in an illustrative way based on the structures of the phase-space portraits: In particular, the conservation properties of real- and imaginary newton-flow with respect to $h$ transfer to the momentum and the sensitivity of the holomorphic flow.

%------------------------------------------------------------
%------------------------------------------------------------
%------------------------------------------------------------
%------------------------------------------------------------

\section{Holomorphic Sensitivity in a Differential Geometry Setting}\label{sec:geomSens}
In this subsection, we introduce a coordinate-in dependent geometric viewpoint connecting the holomoprhic flow, its sensitivities and the proposed Hamiltonian system from the previous Section. The solutions of the sensitivity equation describe dynamics in the tangent bundle $T_{z(t)} \C$ or - alternatively formulated - in $T_{\textbf{z}(t)}\R^2$, when splitting $z$ into real- and imaginary part comprised in the vector $\textbf{z}(t)\in \R^2$. We translate this dynamic to a geometric property within a suitable setting, based on the Hamiltonian function $H(z,p) = h(z)p$. For this, we apply a Legendre transformation to the function
\[
H_{\text{abs}}: \R^4 \to \R, \quad (z_1,z_2,p_1,p_2)\mapsto |h(z_1+iz_2)(p_1+ip_2)|^2
\]
which is evidently constant on the solution surfaces of the Hamiltonian system interpreted in $\R^4$ ($H_{abs}$ contains partial information of $H$).
%\begin{align*}
%&h(z)p \equiv \text{const} \Rightarrow |h(z)p|\equiv \text{const} \\
 %\Rightarrow& \frac{1}{2}|h(z)p|^2  = \frac{1}{2}(p_1^2+p_2^2) (h_1(z_1,z_2)^2+h_2(z_1,z_2)^2) =:H_{\text{abs}}(q_1,q_2,p_1,p_2) %\equiv \text{const} 
%\end{align*}
%with $p_1 = \text{Re}(p)$, $p_2 = \text{Im}(p)$, $z_1 = \text{Re}(z)$, $z_2 = \text{Im}(z)$, $h_1 = \text{Re}(h)$ and $h_2 = \text{Im}(h)$. 
Applying a Legendre transformation to $H_{\text{abs}}$ according to $p_1$ and $p_2$ results in the Lagrangian function
\begin{equation}\label{eq:PMLagrange}
L: \R^2 \times \R^2 \to \R, \quad
(z_1,z_2, \tilde{z}_1,\tilde{z}_2)\mapsto \frac{1}{2}\left(  \frac{\tilde{z}_1^2+\tilde{z}_2^2}{|h(z_1+iz_2)|^2}\right) \quad \text{and} \quad p_k = \frac{\tilde{z}_k}{|h(z_1+iz_2)|^2} 
\end{equation}
with $|h(z_1+iz_2)|^2 = h_1(z_1,z_2)^2+h_2(z_1,z_2)^2$. For a given point $z \in \C$, $h(z) \in \C$, this Lagrangian defines a norm by
$(\tilde{z}_1, \tilde{z}_2) \mapsto \sqrt{L(z_1,z_2,\Tilde{z}_1,\Tilde{z}_2)}$. %This way, $\R^2,\sqrt{L}$ has the structure of a \textit{Finsler manifold}. 
The inducing metric of $\sqrt{L}$ can be expressed by 
\begin{equation}\label{eq:PMmetric}
g_z: \R^2 \times \R^2 \to \R, \quad (v,w) \mapsto \frac{1}{2}\frac{\langle v,w \rangle}{|h(z)|^2}. 
\end{equation}
If $h$ is an entire function, then $(\R^2,g)$ becomes a Riemannian manifold (while the tuple $(\R^2, \sqrt{L})$ has the structure of a Finsler manifold, but this is out of scope for this work). If $h$ is a holomorphic, non-entire function, we may consider the domain $\Bar{U}_h \subset \C$ of $h$ and define 
$U_h := \{(z_1,z_2) \in \R^2~|~\text{s.t.}~z_1+iz_2 \in U_h  \}$
and conclude that $(U_h,g)$ is a Riemannian manifold which we call the $h$-manifold. Applying this construction to the functions $ih,-h$ and $-ih$ yield the same manifold. 

The solutions $z(t)$ of the holomorphic flow (in real time) become smooth curves in the $h$-manifold. The solutions of the sensitivity $\Delta z(z)$ (according to \cref{eq:SensSol}) are smooth vector fields by assigning the section
\begin{equation}\label{eq:hol_splitting}
    (\text{Re}(z),\text{Im}(z)) \mapsto (\text{Re}\left( \Delta z(z) \right), \text{Im}(\Delta z (z))) \in T_z \Bar{U}_h.
\end{equation}
Whenever we construct a vector field this way (by splitting into real- and imaginary part of a holomorphic function), we call it a \textit{holomorphic splitting}. This applies to the functions $h$ and $ih$ in particular.

%In order to avoid confusion regarding the index-placement, 

%We show that both these objects have specific geometric properties when applying the Levi-Civita connection on the $h-$manifold.
%% HIER % On this h mani...  
%In the further parts of this section, we study the geometric properties of these $h-$manifolds and their link to the holomorphic system as well as its sensitivities. As it turns out, studying covariant derivatives provide exactly this 'connection'. For their calculation, we use $z_1$ and $z_2$ as coordinates. 

Let $\partial_k = \partial_{k;z}$ denote the tangent vector in direction $z_k$ (for $k=\{1,2\}$) at $T_z \bar{U}_h$. For this basis of the tangent space, the coordinates of $g_{ij} =g_{ij;z}$ read 
\[
g_{11} = g_{22} = \frac{1}{h_1(z)^2+h_2(z)^2},\quad g_{12} = g_{21} = 0.
\]

\begin{remark}
    We use teletype letters to denote vector fields on the $h$-manifold, such as $\Xttt$ and $\Yttt$. The vector fields defined by the holomorphic splitting of $h$ and $ih$ are denoted by $\httt$ and $\Bar{\httt}$ respectively. We write the vector-valued coordinate functions consisting of real- and imaginary parts in bold letters in this section, e.g. $\textbf{X} = \left(X^1,X^2\right)^T$. Whenever we use a single upper index (without a lower one), it is not meant as an exponent, but only to mark components of tangential vector, consistent with common notation from differential geometry. We partiuclarly receive $h^1(z) = h_1(z)$, where $h_1$ is just the real part of $h$ according to ordinary calculus and $h^1$ is the tensor component of the tangent vector $\httt = \sum_k h^k \partial_k$ in $z_1$-$z_2$-coordinates. 
\end{remark}

The (unique) Levi-Civita connection $\nabla$ of this metric with respect is given can be described by
\begin{lemma}\label{lem:Christoffel}
Let $(z_1,z_2)$ be our fixed coordinates. Let $\Gamma_{ij}^k$ denote the christoffel symbols describing $\nabla$ with respect to the $(z_1,z_2)$ are
\begin{align*}
	\Gamma_{\cdot \cdot}^1 &=
	\frac{1}{h_1(z)^2+h_2(z)^2} \left[
	\begin{array}{rr}
	-m_1 & -m_2\\
	-m_2 & m_1
	\end{array}
	\right], \\  \Gamma_{ \cdot \cdot}^2 &=
	\frac{1}{h_1(z)^2+h_2(z)^2} \left[
	\begin{array}{rr}
	m_2 & -m_1\\
	-m_1 & -m_2
	\end{array}
	\right]
	\end{align*}
	with $m_j:= h_1\frac{\partial h_1}{\partial z_j}+h_2\frac{\partial h_2}{\partial z_j}$ for $j = 1,2$.
 \end{lemma}
 \begin{proof}
     Done in the appendix.
 \end{proof}
Applying \Cref{lem:Christoffel}, we can express the Levi-Civita connection in $z_1$-$z_2$ coordinates. 
 \begin{lemma}\label{lem:CovDeriv}
     Let $\eXttt=X^1(z)\partial_1+X^2(z)\partial_2$ be a holomoprhic splitting with coordinate functions $X^1(z_1,z_2)$ and $X^2(z_1,z_2)$. Let $J_{\emph{\Xbf}}$ and $J_{\emph{\hbf}} \in \R^{2,2}$ denote the jabocian matrices of $\emph{\Xbf}$ and $\emph{\hbf}$ with respect to $z_1$ and $z_2$. Then, the covariant derivatives in direction of $\ehttt$ and $\bar{\ehttt}$ read
	\begin{equation}\label{eq:cov_holsys}
		\nabla_\ehttt \eXttt = \sum_k(J_\emph{\Xbf} \emph{\hbf} - J_\emph{\hbf} \emph{\Xbf})^k \partial_k, \quad \nabla_{\bar{\ehttt}} \eXttt = \sum_k E(J_\emph{\Xbf} \emph{\hbf} - J_\emph{\hbf} \emph{\Xbf})^k \partial_k.
	\end{equation}
	The matrix $E$ is defined by \[E = \left( \begin{array}{rr}
	0&-1 \\1 &0
	\end{array} \right).
	\] 
 \end{lemma}
 \begin{proof}
     Done in the appendix.
 \end{proof}
 When we apply \cref{eq:cov_holsys} to the vector fields $\Xttt = \httt$ and $\Xttt = \Bar{\httt}$, we receive $\nabla_{\httt} \httt = 0$ and $\nabla_{\Bar{\httt}} \Bar{\httt} =0$, meaning that the solution curves and of the real- and imaginary time curves are geodesics with respect to $\nabla$. However, there are additional classes of vector fields $\Xttt$ which are parallel to $\httt$ or $\Bar{\httt}$ respectively:

 \begin{lemma}\label{lem:parallelX}
Let $\emph{\Xttt}$ be a vector field of the form
\[
\eXttt(z_1,z_2) = \sum_k X^k(z_1,z_2) \partial_k \quad \text{with} \quad
 \begin{pmatrix}
 X^1(z_1,z_2) \\
 X^2(z_1,z_2)
 \end{pmatrix} = 
\left( \begin{array}{rr}
a & -b \\b & a
\end{array} \right) \begin{pmatrix}
h_1(z_1,z_2) \\ h_2(z_1,z_2)
\end{pmatrix}
\]
with $a$ and $b$ real numbers. Then, we receive $\nabla_\ehttt \eXttt =0 $ and $\nabla_{\bar{\ehttt}} \eXttt = 0$.  
\end{lemma}
\begin{proof}
Let the matrices
\[A = \left( \begin{array}{rr}
a & -b \\b & a
\end{array} \right)\quad \text{and} \quad E = \left( \begin{array}{rr}
	0&-1 \\1 &0
	\end{array} \right)
 \] be defined as above. Since $h$ is a holomorpic function, $J_\hbf$ has the same from as $A$ (which is evidently true for $E$).  A direct calculation shows that the matrices of the form $A$ and $J_\hbf$ commute (as each represents a holomorphic derivative - which commute). Another calculation shows $J_{A \hbf} = A J_{\hbf}$ (which corresponds to the linearity of the complex derivative: $(ch)' = c h'$ if $c\in \C$ is constant).
 Hence, we can compute
\[
J_{A\hbf} \hbf -J_{\hbf} A\hbf = A J_{\hbf} \hbf- J_{\hbf}A\hbf = (A J_{\hbf} - J_{\hbf} A)\hbf = 0.
\]
Applying \Cref{lem:CovDeriv} implies both $\nabla_\texttt{h} \Xttt =0 $ and $\nabla_{\bar{\texttt{h}}} \Xttt = 0$.
\end{proof}
Evidently, the former lemma can be applied to the holomorphic splitting of the solutions of the sensitivity $\Delta z$. More specifically:
\begin{remark}
Let $z_0 \in U_h$ and $(\Delta z)_0 \in \C$ be arbitrary, then (compare to \cref{eq:SensSol}) the solution of the sensitivity of the flow of $\dot{z} = h(z)$ is
\[
	\Delta z (z)=  \beta h(z) \quad \text{with} \quad \beta= \frac{\Delta z_0}{h(z_0)} \in \C
\]
When representing $z$ and $\Delta z$ by real- and imaginary parts, we receive the function
\[
	\begin{pmatrix}
	(\Delta z (z_1,z_2))_1 \\
	(\Delta z (z_1,z_2))_2
	\end{pmatrix} = A \begin{pmatrix}
	h_1(z_1, z_2)\\
	h_2(z_1,z_2)
	\end{pmatrix} \quad \text{with} \quad A = \begin{pmatrix}
	a& -b \\ b & a  
	\end{pmatrix}
\]
as well as $a= \text{Re}\left( \frac{\Delta z_0}{h(z_0)} \right)$ and $b= \text{Im}\left( \frac{\Delta z_0}{h(z_0)} \right)$. Let $\mathrm{D} =\mathrm{D}(z_0,\Delta z_0)$ be the holomorphic splitting of the solution $\Delta z (z_0,\Delta z_0)$ of the sensitivity equation. Then, we receive
\[
\nabla_\httt \mathrm{D} = 0, \quad \nabla_{\Bar{\httt}} \mathrm{D} =0
\]
on the $h$-manifold. The vectors fields $\httt$ and $\Bar{\httt}$ form a basis of every tangent space $T_z \Bar{U}_h$. Thus, we can also imply $\nabla_{\textstyle Y_z} D =0$ for every point $z$ and tangent vector $Y_z \in T_z$ or simply write $\nabla D =0$. \textit{The holomorphic splittings of the sensitivity are parallel vector fields on the} $h$-\textit{manifold.}
\end{remark}

%{\color{blue} Natural framework. First: tangent and cotangent considerations. Delta z and p (calculated in 3.3 or so) coordinate representations. Metric connects }

 The $h$-manifold hereby provides a 'natural' geometrical framework encoding the central properties of the Hamiltonian system: The differential equations for the sensitivity and the momentum are tangent- and cotangent dynamics respectively, and therefore coordinate-free. The real- and imaginary time solutions of the holomorphic flow are geodesics and the sensitivity solutions for every initial value are parallel vector fields on the $h$-manifold.
 
 The complex-time solution map $t \mapsto z(t)$ of the holomorphic flow of $h$ is a homeomorphism from a small time-disc to an open subset $U \subset\bar{U}_h \subset \C$ of the complex domain $\bar{U}_h$ \cite{Ilyashenko2008} and therefore a coordinate chart. This also applies to the complex-time solution map $T\mapsto z(T)$ of the newton-flow (as long as $h'(z(0)) \ne 0$). These charts are \textit{complexly} compatible and the calculations of \cref{subsec:Xi_newt} describe the (complex $1 \times 1$) transition matrix for the base tangent- and cotangent vectors of these charts:
 \[ \dfrac{\d}{\d T} = \dfrac{\d t}{\d T} \dfrac{\d}{\d t} = -h'(z(t)) \dfrac{\d}{\d t},\qquad 
 \d T = \dfrac{\d T}{\d t} \d t = 
 - \dfrac{1}{h'(z(t))} \d t. \]
These arguments also apply to each Riemann surface of the solution of the Hamiltonian system $H(z,p)$ (see \cref{eq:HamFlow}), defining 'solution charts' $t\mapsto (z(t),p(t))$ and $T\mapsto (z(T),p(T))$ from these small time-discs to open subset of the Riemann surfaces.

 Splitting $t = t_1 + i t_2$ and $T = T_1 +iT_2$ and the solution mapping $z= z_1 + i z_2$ into real- and imaginary parts, we receive two-dimensional real-valued solution mappings from open discs (in $\R^2$) to the $h$-manifold, which are \textit{smoothly} compatible charts for the $h$-manifold. In $(t_1, t_2)$-coordinates, the metric $g$ reads $g_{ij,t} = \delta_{ij}$, with $\delta_{ij}$ being the Delta function (trivial expression of the metric $g$). This especially implies that all $h$-manifolds are curvature-free.

\section{Numerical Studies} 
In this section, we numerically study the Hamiltonian system proposed in \Cref{subsec:HamSyst} and visualize the solution structure of the momentum $p$ as well as the sensitivities $\Delta z$ and $\Delta p$, depending on the state variable $z$. The corresponding solution-formulas are given in \Cref{subsec:Xi_newt,subsec:HamStability} for the $\xi$-function, but directly transfer to the class of Hamiltonian functions $H(z,p) = h(z)p$ with arbitrary holomorphic  $h$ (compare \Cref{sec:HamSens} and \cref{eq:ImpSol,eq:SensSol} in particular). 

\Cref{fig:qp_surf,fig:qdq_surf} show the real-time solutions flows of the $\xi$-approximating polynomial \cref{eq:XiApprPoly} are periodic orbits around each simple roots. 
We inspect one of these orbits (polynomial $h_8$) and exemplify how the solutions $p, \Delta z$ and $\Delta p$ change when the real-time trajectory $z(t)$ transits one orbit cycle. \Cref{fig:qp_surf,fig:qdq_surf} depict absolute values of these quantities (vertical axis). Here, we study their argument.

\subsection*{Sensitivities $\Delta z$ and Momentum $p$ along Closed Orbits}

Recall the solution \cref{eq:ImpSol,eq:SensSol} for $p$ and $\Delta z$. They especially imply that, if $z(t)$ passes through a periodic orbit, then $p$ and $\Delta z$ also behave periodically ($p$ and $\Delta z$ are functions of $z$).

\begin{figure}[bhtp]
\floatbox[{\capbeside\thisfloatsetup{capbesideposition={right,center},capbesidewidth=3.6cm}}]{figure}[\FBwidth]
{\caption{Time propagation of the direction of sensitivity $\Delta z(z;\Delta z_0)$ around periodic orbit of the holomorphic flow of $h_8$ (featuring simple roots only). Different colors indicate different initial values momentum $\Delta z_0$: One tangential- and one normal to the orbit (red- and burgundy-colored respectively). Lengths of arrows are normalized to one.}\label{fig:HamDqTwist}}
{\includegraphics[scale=0.27]{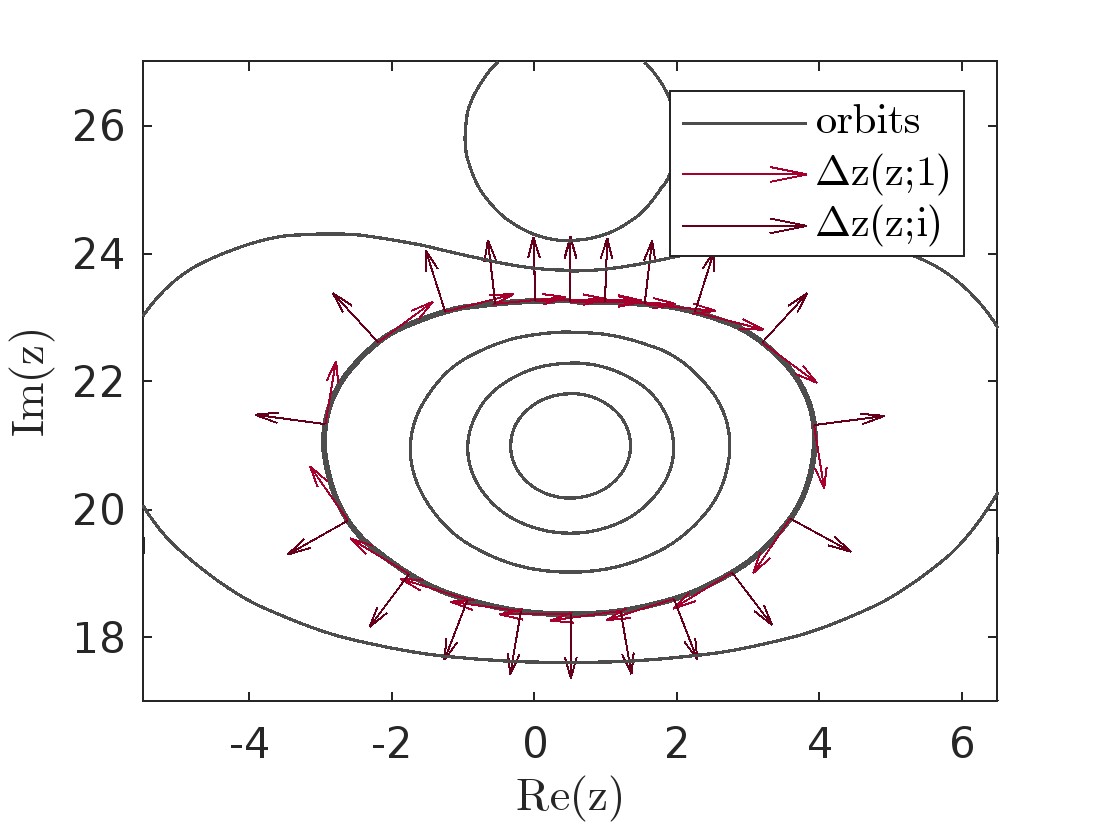}}
\end{figure}

 Rewriting $\Delta z_0 = \alpha h(z_0)$ with  $\alpha = \Delta z_0/h(z_0) \in \C$ and insertion into the solution formula (\cref{eq:SensSol}) leads to
\begin{equation}\label{eq:sens_linear}
    \Delta z(z) = \frac{h(z)}{h(z_0)}\alpha h(z_0) = \alpha h(z).
\end{equation}

This $\C$-linear dependency on $h(z)$ implies that angles between sensitivities (for different initial values $\Delta z_0$) are preserved by time-propagation of the holomorphic flow. 
A particular consequence of 
\cref{eq:sens_linear} is flow-invariance of the tangent- and normal bundle: If $\alpha \in \R$, then $\Delta z$ starts- and remains tangential to the real-time solution trajectory. If $\alpha \in i\R$, the initial vector $\Delta z_0$ starts- and remains in the normal bundle of this solution trajectory.  These findings are exemplified in \Cref{fig:HamDqTwist} for one periodic orbit of the holomorphic real-time flow of the polynomial $h_8$. The red- and burgundy-colored normalized sensitivity-vectors remain tangentially- and normally to the trajectory, twisting ones when transiting this periodic orbit.

\begin{figure}[htbp]
\floatbox[{\capbeside\thisfloatsetup{capbesideposition={right,center},capbesidewidth=3.5cm}}]{figure}[\FBwidth]
{\caption{Time propagation of the direction of momentum variables $p(z;p_0)$ around periodic orbit of of the holomorphic flow of $h_8$ (featuring simple roots only). Different colors indicate different initial values of momentum $p_0$. Lengths of arrows are normalized to one.}\label{fig:HampTwist}}
{\includegraphics[scale=0.25]{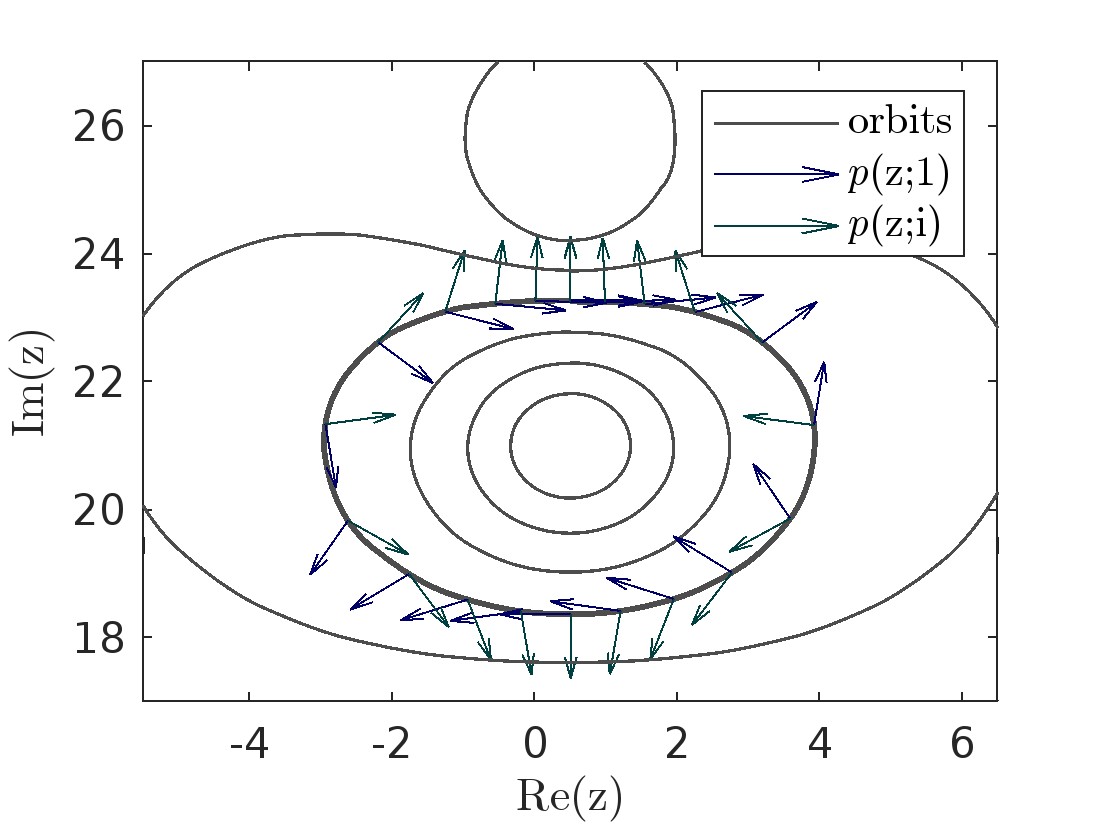}}
\end{figure}

Regarding the momentum $p$, similar aspects are studied: \cref{eq:ImpSol} implies a $\C$-linear dependency of $p$ on the initial value $p_0$. The angles between different initial values $p_0^{(1)}$ and $p_0^{(2)}$ are preserved throughout the orbit:

\begin{equation*}
    \arg\left(p^{(1)}\right)-\arg\left(p^{(2)}\right) \equiv \arg\left(\frac{p^{(1)}}{p^{(2)}}\right)  \equiv  \arg\left(\frac{p_0^{(1)}}{p_0^{(2)}}\right) \equiv \arg\left(p_0^{(1)}\right)-\arg\left(p_0^{(2)}\right) \quad (\text{mod  } \R/2\pi \Z).
\end{equation*}
This behavior is exemplified in \Cref{fig:HampTwist}, showing the example flow as \Cref{fig:HamDqTwist}. Here, we calculate the time propagation of the momentum $p$ for two different perpendicular initial values $p_0$. We see the periodic behavior of $p$ and the preservation of the angle between the different momenta (in blue and green). In contrast to the sensitivity $\Delta z$, the momentum $p$ appear to 'twist' in the inverse rotation direction. This is due to the fact that $p \Delta z$ is constant under the Hamiltonian flow. As  result we receive
\[
\arg\left(\Delta z \right) + \arg \left( p \right) \equiv \arg\left(p \Delta z \right) \equiv \arg\left(p_0\Delta z_0 \right) \equiv
\arg\left(\Delta z_0 \right) + \arg \left( p_0 \right)
\quad (\text{mod  } \R/2\pi \Z).
\]
Hence, as $\Delta z$ turns in on direction (with the flow), $p$ turns against the flow, as shown in \Cref{fig:HampTwist}.
% Covariant/Contravariant?

\subsection*{Second-Order Sensitivity (Momentum Sensitivity) $\Delta p$ along Closed Orbits}%-----------------------------------------%-----------------------------------------%-----------------------------------------
% ---> Dp

Completing the numerical sensitivity analysis of the Hamiltonian flow \cref{eq:HamFlow}, we study the sensitivity of the momentum $p$, denoted by $\Delta p$. Generalizing \cref{eq:DqXiODE} - for the case of the $\xi$-function - the sensitivity equation for $\Delta p$ reads
\[
\dd{t}\Delta p = -h''(z) \Delta z_0 p_0 - h'(z) \Delta p, \qquad \Delta p(z(0)) = \Delta p_0 \in \C.
\]

Due to the connection between $\Delta z$ and $p$ (analysed in \Cref{sec:HamSens} ), $\Delta p$ is linked to the sensitivity $\Delta^2 z := \Delta (\Delta z)$ of $\Delta z$ ('second-order'-sensitivity), also reflected by the appearance of $h''$ in the previous ODE. Generalizing \cref{eq:vardglsol}, we obtain the unique solution
\begin{equation}\label{eq:DpSol}
    \Delta p = \Delta p (z;\Delta z_0,p_0, \Delta p_0) = \underbrace{\frac{(h'(z_0)-h'(z))}{h(z)} p_0 \Delta z_0}_{\textbf{(1)}} + \underbrace{\frac{h(z_0)}{h(z)} \Delta p_0}_{\textbf{(2)}}.
\end{equation}
The solution can be divided into two parts. Part \textbf{(1)} depends on the product of the initial values $p_0$ and $\Delta z_0$ and vanishes for $z = z_0$. In \Cref{fig:HamDpTwist1}, we see how the argument of this first part twists, when passing through a periodic orbit of the state variable $z$ (same flow and orbit as in \Cref{fig:HamDqTwist,fig:HampTwist} using the polynomial $h_8$). $\Delta p_0$ is set to zero in \Cref{fig:HamDpTwist1}. The black dot marks the initial value of the state variable $z_0$. 
The time-propagation sensitivities appears to be discontinuous at $z_0$, which is caused by the fact that we normalized the vectors in these figures and \textbf{(1)} vanishes at $z =z_0$. For both initial values, the argument makes about halve a turn in the direction opposite to the clockwise turning direction of the trajectory $z(t)$. The $\C$-linearity on $p_0$ and $\Delta z_0$ also implies the preservation of angles throughout this orbit (as observed in \Cref{fig:HamDqTwist,fig:HampTwist}).

\begin{figure}[htbp]
\floatbox[{\capbeside\thisfloatsetup{capbesideposition={right,center},capbesidewidth=3.5cm}}]{figure}[\FBwidth]
{\caption{Time propagation of the direction of sensitivity (of $p$) $\Delta p(z;\Delta z_0,p_0,\Delta p_0)$ around periodic orbit of the Hamiltonian system with based of the function $H(z,p)=h_8(z)p$. Different colors indicate different initial values of initial momentum $\Delta z_0$. Lengths of arrows are normalized to one. Black dot marks the initial value of $z_0$. Initial value $\Delta p_0$ is set to zero in both cases.}\label{fig:HamDpTwist1}}
{\includegraphics[scale=0.27]{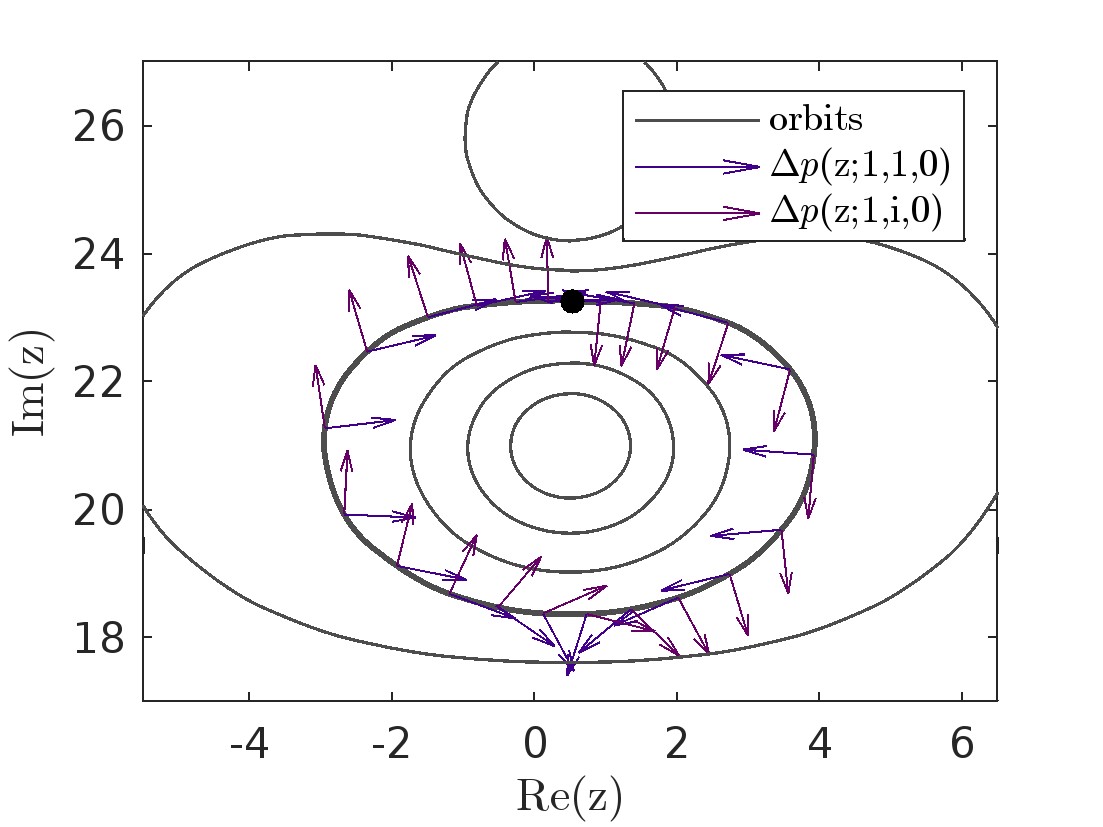}}
\end{figure}

\begin{figure}[htbp]
\floatbox[{\capbeside\thisfloatsetup{capbesideposition={right,center},capbesidewidth=3.5cm}}]{figure}[\FBwidth]
{\caption{Time propagation of the direction of sensitivity (of $p$) $\Delta p(z;\Delta z_0,p_0,\Delta p_0)$ around periodic orbit of the Hamiltonian system with based of the function $H(z,p)=h_8(z)p$. Different colors indicate different initial values of initial momentum $\Delta z_0$. Lengths of arrows are normalized to one. Black dot marks the initial value of $z_0$. Initial value $\Delta p_0$ is \textit{not} set to zero in both cases.}\label{fig:HamDpTwist2}}
{\includegraphics[scale=0.27]{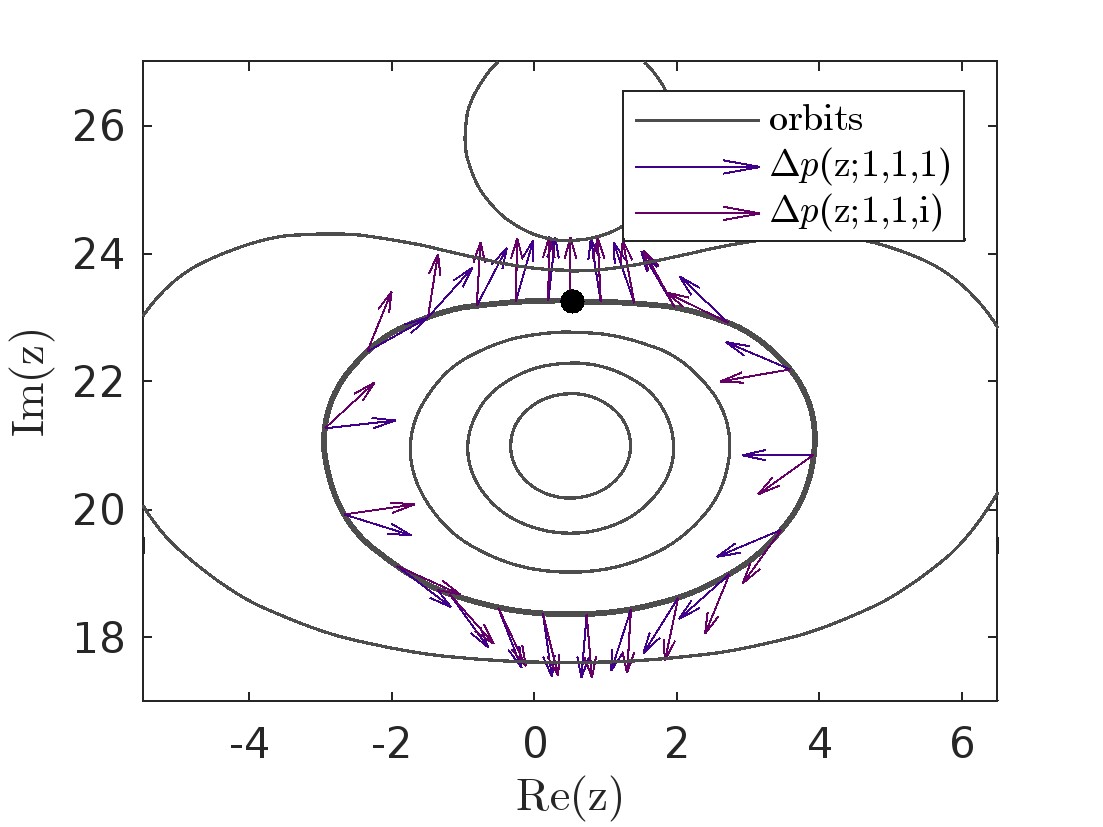}}
\end{figure}

Part \textbf{(2)} of the solution $\Delta p$ only depends $\Delta p_0$. This dependence coincides with $p$ depending on $p_0$ (see \cref{eq:ImpSol}). Hence, its behavior is already shown in \Cref{fig:HampTwist}. The general case, in which all the three initial values $p_0$, $\Delta z_0$ and $\Delta p_0$ don't vanish is depicted in the final \Cref{fig:HamDpTwist2}. Here, we pick a certain combination of initial values for $p_0$ and $\Delta z_0$. We use two different initial values for $\Delta p_0$. Both time-propagations $\Delta p$ twist ones anti-clockwise (similar to \Cref{fig:HampTwist}). This time, we see that the angles between the different solutions for $\Delta p$ change throughout the orbit. This is due to the fact that $\Delta p$ does not linearly depend on $\Delta p_0$ anymore, whenever $p\Delta z \ne 0$.

\section*{Appendix}
\textbf{Proof of \Cref{lem:Christoffel}}: 
\begin{proof}
    Let $z\in U\subset \C$ be fixed and $z_1$ and $z_2$ be the real- and imaginary part of $z$ respectively. The coordinates of the metric tensor with respect to $z_1$ and $z_2$ are indicated by $g_{ij}$. For the sake of shortness, we omit writing the dependence of $h_1$, $h_2$ on $z$ when evaluating these functions. The inverse metric tensor (indicated by $g^{ij}$) is then simply given by taking the inverse matrix of $g_{ij}$: 
    \begin{equation}\label{eq:invTensor}
         g_{ij} = \left(\frac{1}{h_1^2+h_2^2}\right)  \text{Id}_2, \quad g^{ij} = \left(h_1^2+h_2^2\right)\text{Id}_2
    \end{equation}
    with $\text{Id}_2$ representing the two-by-two identity matrix. We indicate the Christoffel symbols of the first- and second kind by $\Gamma_{k,ij}$ and $\Gamma_{ij}^k$ respectively. The indices $i$ and $j$ are the symmetrical ones. 
    We use the Koszul formula to calculate the Christoffel symbols:
    \begin{align*}
        \Gamma_{k,ij} &= \frac{1}{2} \left(\frac{\partial g_{ik}}{\partial z_j} + \frac{\partial g_{jk}}{\partial z_i} - \frac{\partial g_{ij}}{\partial z_k} \right)  \\
        \Gamma^k_{ij} &= \sum_{l = 1,2} g^{kl}\Gamma_{l,ij}
    \end{align*}
        for all $(i,j,k) \in \{1,2\}^3$.
    Note that
    \begin{align*}
        \frac{\partial g_{11}}{\partial z_l} = \frac{\partial g_{22}}{\partial z_l} &= -\frac{1}{(h_1^2+h_2^2)^2} \frac{\partial}{\partial z_l} \left( h_1^2+h_2^2 \right) \\
        &= -\frac{2}{(h_1^2+h_2^2)^2} \underbrace{\left( h_1 \frac{\partial h_1}{\partial z_l} +h_2 \frac{\partial h_2}{\partial z_l} \right).}_{\textstyle := m_l} 
    \end{align*}
    The Christoffel symbols of the first kind then read
    \begin{align*}
        \Gamma_{1,\cdot \cdot} &= \frac{1}{2} \left(
        \begin{array}{rr}
             \dfrac{\partial g_{11}}{\partial z_1} & \dfrac{\partial g_{11}}{\partial z_2} \\[2ex]
             \dfrac{\partial g_{11}}{\partial z_2}& -\dfrac{\partial g_{22}}{\partial z_1}\\ [2ex]
        \end{array}
        \right)
        = \frac{1}{(h_1^2+h_2^2)^2}
        \left(
        \begin{array}{rr}
             -m_1 & -m_2 \\
             -m_2& m_1\\ 
        \end{array}
        \right) \\
        \Gamma_{1,\cdot \cdot} &= \frac{1}{2} \left(
        \begin{array}{rr}
             -\dfrac{\partial g_{11}}{\partial z_2} & \dfrac{\partial g_{22}}{\partial z_1} \\[2ex]
             \dfrac{\partial g_{22}}{\partial z_1}& \dfrac{\partial g_{22}}{\partial z_2}\\ [2ex]
        \end{array}
        \right)
        = \frac{1}{(h_1^2+h_2^2)^2}
        \left(
        \begin{array}{rr}
             m_2 & -m_1 \\
             -m_1& -m_2\\ 
        \end{array}
        \right)
    \end{align*}
    Because of the form of the inverse metric tensor in \cref{eq:invTensor}, the Koszul formula implies $\Gamma_{ij}^k = (h_1^2+h_2^2)\Gamma_{k,ij}$ for all $(i,j,k) \in \{1,2\}^3$, proving the hypothesis.
\end{proof}
\textbf{Proof of \Cref{lem:CovDeriv}:}
\begin{proof}
    Recall that the Christoffel symbols of the second kind calculated in \Cref{lem:Christoffel} are the unique solutions of the system of equations
    \[
    \nabla_{\textstyle \partial_i}\partial_j = \sum_k \Gamma_{ij}^k \partial_k \quad \forall (i,j) \in \{1,2\}^2.
    \]
    Let $\Xttt$ be a holomorphic splitting and $\Yttt$ be an arbitrary vector field with coordinate functions (in $z$-coordinates) $X^j$ and $Y^i$, the linearity, product rule and last equality yield
    \begin{align} \label{eq:covDerXY}
    \begin{split}
         \nabla_{\textstyle \Yttt} \Xttt &= \sum_{i,j} Y^j \nabla_{\textstyle\partial_j}X^i \partial_i = \sum_{i,j} Y^j X^i \nabla_{\textstyle\partial_i}\partial_j + Y^j\partial_j(X^i)\cdot \partial_i \\
         &= \sum_k \underbrace{\left(  \sum_{i,j} Y^j X^i \Gamma_{ij}^k \right)}_{\textstyle\textbf{(I)}} \partial_k + \sum_k \underbrace{\left( \sum_j Y^j\partial_j(X^k) \right)}_{\textstyle\textbf{(II)}}\partial_k
    \end{split}
    \end{align}
    We see that \textbf{(II)} can be rewritten using matrix multiplication. Let $\Xbf = (X^1,X^2)^T$ and $\Ybf = (Y^1,Y^2)^T$ be the vectors of the component functions. Then,
    \begin{align*}
        \sum_j Y^j\partial_j(X^k) = \text{grad}\left(X^k(z_1,z_2)\right) \Ybf \quad \Rightarrow \sum_k\left( \sum_j Y^j\partial_j(X^k) \right)\partial_k = \left(J_{\Xbf}\Ybf\right)^T \begin{pmatrix}
            \partial_1\\ \partial_2
        \end{pmatrix}
    \end{align*}
    This former equality is applied to the vector fields $\httt$ and $\Bar{\httt}$ instead of $\Yttt$. Let $\hbf$ and $\Bar{\hbf} = E \hbf $ denote the respective vector functions ($E$ is defined in \Cref{lem:CovDeriv} and used below). Applying these vector fields leads to
    \begin{equation} \label{eq:(II)}
        \sum_k\left( \sum_j h^j\partial_j(X^k) \right)\partial_k = \left(J_{\Xbf}\hbf \right)^T \begin{pmatrix}
            \partial_1\\ \partial_2 \end{pmatrix}, \quad 
            \sum_k\left( \sum_j \Bar{h}^j\partial_j(X^k) \right)\partial_k = E \left(J_{\Xbf}  \hbf \right)^T \begin{pmatrix}
            \partial_1\\ \partial_2 \end{pmatrix}.
    \end{equation}
    In the last equality we used the fact that the matrices $J_{\Xbf}$ and $E$ commute. This is true, since $X$ is a holomorphic splitting, and therefore the Cauchy-Riemann equations hold, implying 
    \[
    J_\Xbf = \begin{pmatrix}
        a &-b\\b & a
    \end{pmatrix}
    \quad \text{with } a = \dfrac{\partial X^1}{\partial z_1} =\dfrac{\partial X^2}{\partial z_2}\quad \text{and } b = \dfrac{\partial X^2}{\partial z_1} = - \dfrac{\partial X^1}{\partial z_2}.
    \]
    A direct calculation shows that $E$ and $J_\Xbf$ commute. We now calculate \textbf{(I)}. If we insert $\httt$ instead of $\Yttt$ with component functions $h^j(z)$, we receive 
    \begin{align*}
        \sum_{i,j} h^j X^i \Gamma_{ij}^1 &= \frac{1}{(h_1^2+h_2^2)} \left( (h_1,h_2) \left(\begin{array}{rr}
         -m_{z_1}& -m_{z_2}  \\
         -m_{z_2}& m_{z_1} 
    \end{array} \right) \begin{pmatrix}
        X_1\\X_2
    \end{pmatrix} \right) \\
    &=\frac{1}{(h_1^2+h_2^2)} \left( (X_2h_2 -X_1h_1) \cdot \left(h_1 \dfrac{\partial h_1}{\partial z_1} + h_2 \dfrac{\partial h_2}{\partial z_1}\right)
     \right) \\
     &-\frac{1}{(h_1^2+h_2^2)} \left( (X_2h_1 +X_1h_2) \cdot \left(h_1 \dfrac{\partial h_1}{\partial z_2} + h_2 \dfrac{\partial h_2}{\partial z_2}\right)
     \right)\\
     &\overset{(\star)}{=} \frac{1}{(h_1^2+h_2^2)} \left( \dfrac{\partial h_1}{\partial z_1} \left(h_1\left(h_2X_2-h_1X_1 \right)-h_2\left(h_1X_2+h_2X_1 \right)  \right)\right) \\
     &+ \frac{1}{(h_1^2+h_2^2)} \left( \dfrac{\partial h_1}{\partial z_2} \left(-h_1\left(h_1X_2+h_2X_1 \right)+h_2\left(h_1X_1-h_2X_2 \right)  \right)\right) \\
     &=- (\text{grad }h_1(z_1,z_2))^T \Xbf
    \end{align*}
    where we used the Cauchy-Riemann equations on the derivatives of $\hbf$ in the step marked by $(\star)$.
    Similarly, for the second component $(k=2)$, we calculate
    \begin{align*}
        \sum_{i,j} h_j X_i \Gamma_{ij}^2 &= \frac{1}{(h_1^2+h_2^2)} \left( (h_1,h_2) \left(\begin{array}{rr}
         m_{z_2}& -m_{z_1}  \\
         -m_{z_1}& m_{z_2} 
    \end{array} \right) \begin{pmatrix}
        X_1\\X_2
    \end{pmatrix} \right) \\
    &=\ldots =   - (\text{grad }h_2(z_1,z_2))^T \Xbf
    \end{align*}
    Combining the last two calculations and applying them for the term (I) for $Y = \httt$, we receive
    \begin{equation}\label{eq:(I)}
        \sum_k \left(  \sum_{i,j} h_j X_i \Gamma_{ij}^k \right) = -\sum_k\left(J_{\hbf} \Xbf\right)_k \partial_k
    \end{equation}
    Combining the former equality with \cref{eq:(II)} implies the hypothesis for the vector field $\httt$. We can trace back the case $Y = \Bar{\httt}$ to the previous case by making the following argument:

    Consider the holomorphic function $ih$. Then, the $ih$-manifold and the $h$-manifold are equal: They share the same topology, structure and metric (since $|ih|^2 = |h|^2$). The holomorphic splitting of $ih$ is then exactly $\Bar{\httt}$ with $\Bar{h}_1=-h_2$ and $\Bar{h}_2=h_1$ representing the real- part and imaginary part respectively. 
    When applying \cref{eq:(I)} to $\Bar{\httt}$, we receive 
    \[
    \sum_k \left(  \sum_{i,j} \Bar{h}_j X_i \Gamma_{ij}^k \right) = -\sum_k\left(J_{\Bar{\hbf}} \Xbf\right)_k \partial_k =- \sum_k\left(E J_{\hbf} \Xbf\right)_k \partial_k
    \]
    where we used the fact that $J_{\Bar{\hbf}} = E J_{\hbf}$. Thus, the hypothesis is shown.
 \end{proof}
\section*{Acknowledgments}
The authors thank Marcus Heitel for technical support.

%Bibliography
\bibliographystyle{unsrt}  
\bibliography{PaperLiterature}

\end{document}